\documentclass[reqno]{amsart}

\usepackage{amsmath,amsthm}
\usepackage{dsfont} 
\usepackage[english]{babel}
\usepackage{color}
\usepackage{amssymb}
\usepackage{mathrsfs}
\usepackage{graphicx}
\usepackage[font=footnotesize, labelfont=bf]{caption}
\usepackage{mathtools}
\usepackage[colorlinks=true, pdfstartview=FitV, linkcolor=blue, citecolor=blue, urlcolor=blue,pagebackref=false]{hyperref}
\usepackage{microtype}
\usepackage{enumitem}
\usepackage[normalem]{ulem}
\usepackage{soul}

\usepackage{float}

\newtheorem{theorem}{Theorem}[section]
\newtheorem{proposition}[theorem]{Proposition}
\newtheorem{lemma}[theorem]{Lemma}

\theoremstyle{definition}
\newtheorem{remark}[theorem]{Remark}

\newtheorem{definition}[theorem]{Definition}

\newtheorem{assumption}{Assumption}
\numberwithin{equation}{section}

\definecolor{myblue}{RGB}{0,0,128}

\newcommand{\LM}[1]{\hbox{\vrule width.2pt \vbox to#1pt{\vfill \hrule
width#1pt
height.2pt}}}
\newcommand{\LL}{{\mathchoice {\>\LM7\>}{\>\LM7\>}{\,\LM5\,}{\,\LM{3.35}\,}}}

\def\Xint#1{\mathchoice
	{\XXint\displaystyle\textstyle{#1}}%
	{\XXint\textstyle\scriptstyle{#1}}%
	{\XXint\scriptstyle\scriptscriptstyle{#1}}%
	{\XXint\scriptscriptstyle\scriptscriptstyle{#1}}%
	\!\int}
\def\XXint#1#2#3{{\setbox0=\hbox{$#1{#2#3}{\int}$ }
		\vcenter{\hbox{$#2#3$ }}\kern-.57\wd0}}

\def\dashint{\Xint-}

\newcommand{\dx}{\,\mathrm{d}x}
\newcommand{\dy}{\,\mathrm{d}y}

\newcommand{\e}{\varepsilon}
\newcommand{\dist}{{\rm{dist}}}

\newcommand{\w}{\omega}

\newcommand{\R}{\mathbb{R}}
\newcommand{\Z}{\mathbb{Z}}
\newcommand{\N}{\mathbb{N}}
\newcommand{\Q}{\mathbb{Q}}

\renewcommand{\a}{\alpha}

\newcommand{\A}{\mathcal{A}}

\newcommand{\F}{\mathcal{F}}

\renewcommand{\P}{\mathbb{P}}
\newcommand{\D}{\mathrm{D}}
\newcommand{\Am}{\Lambda}

\newcommand{\ie}{\,{i.e.}, }

\begin{document}

\author{Matthias Ruf}
\address[Matthias Ruf]{Section de math\'ematiques, Ecole Polytechnique F\'ed\'erale de Lausanne, Station 8, 1015 Lausanne, Switzerland}
\email{matthias.ruf@epfl.ch}

\author{Caterina Ida Zeppieri}
\address[Caterina Ida Zeppieri]{Angewandte Mathematik, Universit\"at M\"unster, Einsteinstra\ss e 62, 48149 M\"unster, Germany}
\email{caterina.zeppieri@uni-muenster.de}

\title[Stochastic homogenization of degenerate integral functionals with linear growth]{Stochastic homogenization of degenerate integral functionals with linear growth}

\begin{abstract}	
We study the limit behaviour of a sequence of non-convex, vectorial, random integral functionals, defined on $W^{1,1}$, whose integrands
satisfy degenerate linear growth conditions. These involve suitable random, 
scale-dependent weight-functions. Under minimal assumptions on the integrand and on the weight-functions, we show that the sequence of functionals homogenizes to a non-degenerate functional defined on $BV$. \end{abstract}
\maketitle
{\small
	\noindent
	
	\noindent
}	


\section{Introduction}
The stochastic homogenization of non-degenerate integral functionals is by now well-understood. The first result in the nonlinear Sobolev setting dates back to \cite{DMMoI,DMMoII}, where, using the language of $\Gamma$-convergence, Dal Maso and Modica analyse the limit behaviour of sequences of random integral functionals depending on a small parameter $\e>0$ and satisfying standard growth conditions of order $p>1$. More precisely, for a given a complete probability space $(\Omega, \F, \mathbb P)$, Dal Maso and Modica consider functionals of the type
\begin{equation}\label{eq:F_eps}
	F_{\e}(\w)(u)=\int_A f(\w,\tfrac{x}{\e},\nabla u)\dx,
\end{equation}
where $A\subset \R^d$ is an open, bounded, Lipschitz set, $f \colon \Omega \times \R^d \times \R^d \to [0,+\infty)$ is measurable in $(\w,x)$, convex in the gradient-variable, and for every $(\omega,x,\xi)\in \Omega \times \R^d \times \R^d$ satisfies
\begin{equation}\label{eq:p-growth}
	\alpha |\xi|^p\leq f(\w,x,\xi)\leq \beta (|\xi|^p+1),
\end{equation}
for $p>1$ and $\alpha, \beta>0$. In \eqref{eq:F_eps} the functionals $F_\e$ depend on the random parameter $\omega\in \Omega$ through $f$, which is then to be interpreted as an \emph{ensemble} of integrands. This means that in this setting only the statistical specification of $f$ is known. 

For homogenization to take place $f$ needs to be \emph{stationary} or \emph{periodic in law}, which amounts to saying that the statistics of $f$ are translation invariant. The stationarity of $f$ can be quantified in terms of a measure-preserving group-action $\{\tau_z\}_{z\in \R^d}$ defined on $(\Omega, \F, \mathbb P)$ by requiring that  
\begin{equation}\label{intro:stationarity}
	f(\omega,x+z,\xi)= f(\tau_z \omega,x,\xi),
\end{equation}
for every $z\in \R^d$, and for every $(\omega,x,\xi)\in \Omega \times \R^d \times \R^d$. We notice that \emph{periodicity} is a particular instance of stationarity. Indeed choosing $\Omega=[0,1)^d$ and $\mathbb P = \mathcal L^d|_{[0,1)^d}$, we 
have that $\tau_z \omega=\omega+z$ (mod $1$) defines a  $\mathbb P$-preserving group-action on $\Omega$ (cf. Definition \ref{def:group-action}). 
Then any $(0,1)^d$-periodic function $g\colon \Omega \times \R^d\to [0,+\infty)$ corresponds to the stationary $f$ given by
\[
f(\w,x,\xi)=g(\w+x,\xi).
\]
It is easy to check that \emph{quasi-periodicity} is a special case of stationarity as well (see e.g. \cite{JKO}). 

Under the assumptions as above, Dal Maso and Modica \cite{DMMoII} prove that the random functionals $F_{\e}(\w)$ $\Gamma$-converge \emph{almost surely} to a random functional of the form
\begin{equation}\label{intro:F-hom-p}
	F_{\rm hom}(\w)(u)=\int_A f_{\rm hom}(\w,\nabla u)\dx,
\end{equation}
where $f_{\rm hom}$ satisfies \eqref{eq:p-growth} (with the same constants $\a,\beta$) and $\mathbb P$-a.s. and for every $\xi \in \R^d$ is given by
\begin{equation}\label{intro:f-hom-p}
	f_{\rm hom}(\omega,\xi)=\lim_{t\to +\infty}\frac{1}{t^d}\inf\left\{\int_{Q_t(0)}f(\w, x, \nabla u+\xi)\dx :\,u\in W^{1,p}_0(Q_t(0))\right\},
\end{equation}
where $Q_t(0)$ is the open cube of $\R^d$ centred at zero and with side-length $t$.
Moreover, under the additional assumption of \emph{ergodicity}, which loosely speaking means that the statistics of $f$ decorrelate over large distances, the integrand $f_{\rm hom}$ is actually deterministic and is given by the expected value of the random variable $f_{\hom}(\cdot, \xi)$.    

Although the formula defining $f_{\rm hom}$ is formally analogous to the asymptotic cell-formula of periodic homogenization (cf. \cite{BrDf}), the reason why the limit in \eqref{intro:f-hom-p} exists (almost surely) and defines a spatially homogeneous quantity is nontrivial, contrary to the deterministic periodic case.  In fact, as observed for the first time by Dal Maso and Modica in their seminal work \cite{DMMoII}, the well-posedness and $x$-homogeneity of \eqref{intro:f-hom-p} in this case is obtained by showing that, for fixed $\xi\in \R^d$, the map     
\[
(\omega, A) \mapsto \inf\left\{\int_{A}f(\w, x, \nabla u +\xi)\dx :\,u\in W^{1,p}_0(A)\right\},
\]
defines a so-called \emph{subadditive process} on $\Omega \times \A$ (cf. Definition \ref{def:sub-p}), where $\A$ denotes the class of open, bounded, Lipschitz subsets of $\R^d$, and then invoking the pointwise ergodic Theorem of Akcoglu and Krengel \cite{AkKr}. 
Once the existence of $f_{\rm hom}$ is established, the homogenization result for $F_\e(\omega)$ can be proven either appealing to the integral representation result in \cite{DMMoIII} or using a more modern approach based on the Fonseca and M\"uller blow-up method \cite{FoMue}.   

The Dal Maso and Modica proof-strategy is flexible enough to be adapted to the vectorial setting both in the case of superlinear ($p>1$)\cite{MeMi} and of linear ($p=1$) \cite{AML} standard growth conditions in the gradient variable.   
However, we notice here that when $f$ grows linearly in the gradient variable, the homogenized functional $F_{\rm hom}$ has a different structure with respect to the functional in \eqref{intro:F-hom-p} (see \cite{AML}). In fact, in the linear case sequences $(u_\e) \subset W^{1,1}(A)$ with equi-bounded energy $F_\e(\w)$ are precompact only in $BV(A)$, the \emph{space of functions of bounded variation} and therefore in this case homogenization and relaxation occur simultaneously. Then, the limit functional is of the form
\begin{equation}\label{intro:F-hom-1}
	F_{\rm hom}(\w)(u)=\int_A f_{\rm hom}(\w,\nabla u)\dx+\int_A f_{\rm hom}^{\infty}\left(\w,\frac{\mathrm{d}\D^su}{\mathrm{d}|\D^su|}\right)\,\mathrm{d}|\D^su|,
\end{equation}
where $f_{\rm hom}$ is given by \eqref{intro:f-hom-p} choosing $p=1$ and $f_{\rm hom}^{\infty}$ denotes the recession function of $f_{\rm hom}$, that is, the slope of $f_{\hom}$ at infinity. Moreover, we recall that $\nabla u \, {\rm d}x$ and ${\rm D}^s u$ represent, respectively, the absolutely continuous and the singular part of the (finite) measure ${\rm D}u$ with respect to the Lebesgue measure in $\R^d$. Loosely speaking, \eqref{intro:F-hom-1} shows that in the linear setting the homogenization can be performed before the relaxation takes place. We also refer to \cite{CDMSZ21} for analogous effects when the heterogeneous random functionals $F_{\e}(\omega)$ are defined on the space of special functions of bounded variation, $SBV$, instead of $W^{1,1}$. 

Recently the homogenization theory of random integral functionals of type \eqref{eq:F_eps} has been extended to the case of \emph{degenerate integrands}. By degenerate integrands we mean that $f$ satisfies growth conditions where the gradient variable is weighted by a non-homogeneous, random coefficient which is not necessarily bounded. This then leads to ``nonstandard'' upper and lower bounds of the form 
\begin{equation}\label{eq:deg-growth-p}
	\alpha |\xi \Lambda (\w,x)|^p\leq f(\w,x,\xi)\leq |\xi \Lambda(\w,x)|^p+\lambda(\w,x),
\end{equation}
where $\Lambda$ is a diagonal matrix-valued function, and both $\Lambda$ and $\lambda$ are stationary. 

The case $p>1$ has been firstly studied in \cite{NSS} where the authors establish a homogenization result for \emph{discrete} functionals 
of the form
\begin{equation*}
	F_{\e}(\w)(u)=\sum_{z\in\e\Z^d\cap A}\sum_{b\in\mathcal{E}_0}f_b\left(\w,\tfrac{z}{\e},\partial_b^{\e} u(z)\right),
\end{equation*}
where the variable $u$ takes values in $\R^m$, $\mathcal{E}_0$ is a finite set of interaction-edges, and $\partial_b^{\e}$ represents a discrete derivative along the scaled edge $(z,z+\e b)$. The stationary, ergodic interaction potentials $f_b$ are assumed to satisfy an estimate as in \eqref{eq:deg-growth-p}, but with a different scalar-valued weight-function $\Lambda_b$ for each edge $b\in\mathcal{E}_0$ (which corresponds to a diagonal matrix $\Lambda$ if one considers edge derivatives as partial derivatives) satisfying the moment condition
\begin{equation}\label{eq:moment-condition-p}
	\mathbb{E}[|\Lambda_b(\cdot,0)|^p]+\mathbb{E}[|\Lambda_b^{-1}(\cdot,0)|^{-p/(p-1)}]<+\infty.
\end{equation}
Moreover, in the scalar case $m=1$ an additional convexity assumption at infinity is imposed, while in the vectorial case $m>1$ the proof relies on the integrability assumption
\begin{align}\label{intro:Neu-cond}
	\mathbb{E}[|\Lambda_b(\cdot,0)|^{r p}]+\mathbb{E}[|\Lambda_b^{-1}(\cdot,0)|^{-sp}]&<+\infty\quad\text{ for some } r>1\text{ and $s$ such that }\frac{1}{r}+\frac{1}{s}\leq \frac{p}{d},
\end{align}
which is strictly stronger than \eqref{eq:moment-condition-p}.
Under these assumptions the $\Gamma$-limit of $F_\e(\w)$ is also of type \eqref{intro:F-hom-p} (and deterministic due to ergodicity) and the homogenized integrand satisfies standard $p$-growth conditions, similarly as in the non-degenerate case. 

In the more recent \cite{RR21} the $\Gamma$-convergence of $F_\e(\omega)$ as in \eqref{eq:F_eps}, with integrand satisfying \eqref{eq:deg-growth-p}, is proven under the sole integrability condition \eqref{eq:moment-condition-p}.  In this case the $\Gamma$-limit is a non-degenerate homogeneous functional of the form \eqref{intro:F-hom-p}, moreover condition \eqref{eq:moment-condition-p}  is shown to be the optimal one to obtain a non-degenerate limit integrand. To avoid the more restrictive condition in \eqref{intro:Neu-cond}, in \cite{RR21} the authors rely on a vectorial truncation-argument combined with an \emph{ad hoc} variant of the Birkhoff additive Ergodic Theorem (cf. Theorem \ref{thm.additiv_ergodic} in Section \ref{s.preliminaries}). In the proof of the vectorial truncation-result, it is of crucial importance that the matrix-valued coefficient $\Lambda$ has a diagonal structure. We observe that this choice still allows to cover the case of anisotropically degenerate integrands. 

Finally, in \cite{D'OnZe} the $\Gamma$-convergence of general integral functionals defined on scale-dependent \emph{weighted Sobolev Spaces} is analysed without requiring any stationarity of the integrands. In this case, under suitable uniform integrability assumptions on the (scalar) weight functions $\Lambda_\e$, which also need to belong to a Muckenhoupt class, the functionals are shown to $\Gamma$-converge (up to subsequences) to a degenerate integral functional defined on a ``limit'' weighted Sobolev space.

\medskip

In the present paper we consider random integral functionals $F_\e(\w)$ of type \eqref{eq:F_eps} with integrand $f$ satisfying degenerate \emph{linear} growth conditions of type
\begin{equation}\label{eq:deg-growth}
	c|\xi \Lambda(\w,x)|\leq f(\w,x,\xi)\leq |\xi \Lambda(\w,x)|+\lambda(\w,x),
\end{equation}
where the stationary functions $\Lambda$ and $\lambda$ satisfy the moment conditions 
\begin{equation}\label{intro:int-1}
	|\Lambda (\cdot,0)|,\lambda(\cdot,0) \in L^1(\Omega) \quad \text{and} \quad |\Lambda (\cdot,0)^{-1}|\in L^{\infty}(\Omega), 
\end{equation}
with $\Lambda$ being a diagonal matrix-valued function.

Besides joint measurability, the only regularity assumption we require for the realisations of the random integrand $f$ is the lower semicontinuity in the $\xi$-variable.
Then, under stationarity of $f$ and of the coefficient-functions $\Lambda, \lambda$ (cf. Assumption \ref{a.1}) in Theorem \ref{thm.Gamma_pure} we prove that, almost surely, the functionals $F_{\e}(\w)$ homogenize to a random functional $F_{\rm hom}(\omega)$ which is spatially homogeneous and deterministic if ergodicity is additionally assumed. Similarly as in \cite{AML}, the limit functional is finite on $BV(A,\R^m)$ where it is of the same form as \eqref{intro:F-hom-1} with $f_{\rm hom}$ given by
\begin{equation}\label{intro:f-hom-1}
	f_{\rm hom}(\omega,\xi)=\lim_{t\to +\infty}\frac{1}{t^d}\inf\left\{\int_{Q_t(0)}f(\w, x, \nabla u+\xi)\dx :\,u\in W^{1,1}_0(Q_t(0),\R^m)\right\}.
\end{equation}
Moreover, in the ergodic case there holds
\begin{equation*}
	f_{\rm hom}(\xi)=\mathbb E[f_{\rm hom}(\cdot,\xi)]= \lim_{t\to +\infty}\frac{1}{t^d}\mathbb E\bigg[\inf\bigg\{\int_{Q_t(0)}f(\cdot, x, \nabla u+\xi)\dx :\,u\in W^{1,1}_0(Q_t(0),\R^m)\bigg\}\bigg],
\end{equation*}
with $f_{\rm hom}$ satisfying the following standard linear growth conditions 
\begin{equation}\label{intro:gc-lim}
	\alpha\,c_0|\xi|\leq f_{\rm hom}(\xi)\leq C_0|\xi|+C_1,
\end{equation}
with constants 
\[
\text{$c_0:=\||\Lambda(\cdot,0)^{-1}|\|_{L^{\infty}(\Omega)}^{-1}$,\quad $C_0:=\sup_{\eta \in \R^{m\times d}, |\eta|=1}\mathbb{E}[|\eta \Am(\cdot,0)|]$,\quad and $C_1:=\mathbb{E}[\lambda(\cdot,0)]$.}
\]
It is worth noticing already here that the integrability conditions in \eqref{intro:int-1} are the optimal ones for \eqref{intro:gc-lim} to hold true. Namely, in Remark \ref{r.optimal} we show that the upper bound in \eqref{intro:gc-lim} implies the finite first moment condition $|\Lambda (\cdot,0)| \in L^1(\Omega)$, while the bound from below in \eqref{intro:gc-lim} is violated if $|\Lambda (\cdot,0)^{-1}|\notin L^{\infty}(\Omega)$. In fact, in this case there is a loss of $BV$-compactness which makes it possible to approximate interfaces at no cost. 

The proof of the stochastic homogenization result, Theorem \ref{thm.Gamma_pure}, follows the general strategy of Dal Maso and Modica \cite{DMMoII}. Namely, the asymptotic homogenization formula in \eqref{intro:f-hom-1} is established by applying a suitable variant of the  Akcoglu and Krengel pointwise ergodic Theorem (cf. \cite[Theorem 4.3]{LICHT-MICHAILLE}). Then, the almost sure $\Gamma$-convergence of the random functionals $F_{\e}(\w)$ is proven in two main steps. In the first step we show that the upper-bound inequality for the $\Gamma$-limit holds true. The proof of this inequality combines an explicit construction of a recovery sequence for $W^{1,1}$-target functions with a relaxation argument. Although the construction of the recovery sequence is rather classical, in this degenerate setting a new argument is needed to construct sequences $(u_\e)$ matching the right linear boundary conditions, in order to reconstruct the limit functional $F_{\rm hom}(\w)$ from $F_{\e}(\w)(u_\e)$. This argument is based on a standard vectorial-truncation result (cf. Lemma \ref{l.truncation}) and a new, \emph{ad hoc}, fundamental estimate for random degenerate integral functionals (see Lemma \ref{l.fundamental_estimate} and Remark \ref{rmk:sf}).  

The proof of the lower bound inequality for the $\Gamma$-limit is more delicate and is based on an adaptation to the $BV$-setting (cf. \cite{AmDM92}) of the Fonseca and M\"uller blow-up method \cite{FoMue}. In this case for any sequence $(u_\e)\subset W^{1,1}(A,\R^m)$ with equi-bounded energy and such that $u_\e \to u$ in $L^1(A,\R^m)$, the set function $\nu_\e(\omega, \cdot):=F_\e(\omega)(u_\e,\cdot)$ is interpreted as a (random) Radon measure on $\A$. By assumption, up to subsequences, $\nu_\e \to \nu$, for some limit (random) Radon measure $\nu$. Then, if we write 
\[
\nu(\omega,A)=\int_A \tilde f(\omega, x)\dx + \nu^s(\omega,A),       
\]
with $\nu^s \perp {\rm d}x$,  the idea of the blow-up method is to perform a local analysis to establish that, almost surely, the following two inequalities hold true
\begin{equation}\label{intro:1}
	\tilde f(\omega, x) \geq f_{\rm hom}(\omega, \nabla u(x))\quad \text{for a.e. $x\in A$}, 
\end{equation}
\begin{equation}\label{intro:2}
	\frac{{\rm d}\nu^s(\omega,x)}{{\rm d}\D u}\geq f^\infty_{\rm hom}\Big(\omega,\frac{{\rm d}\D u}{{\rm d} |\D u|}(x)\Big)\;\text{ for }|\D^s u|\text{-a.e. }x\in A.
\end{equation}
Eventually the lower-bound inequality for the $\Gamma$-limit follows by integrating \eqref{intro:1} and \eqref{intro:2}. 
\medskip

{\bf Structure of the paper.} This paper is organised as follows. In Section \ref{s.preliminaries} we introduce some notation, collect some useful facts on $BV$-functions, and recall some basic Ergodic Theory. In Section \ref{s.results} we state the main result of this paper, Theorem \ref{thm.Gamma_pure}, which establishes an almost sure homogenization result for the functionals $F_\e(\omega)$ under linear degenerate growth conditions. In this section we also discuss the optimality of our assumptions (see Remark \ref{r.optimal}). 
Eventually, in Section \ref{s.results}, a homogenization theorem for the functionals $F_\e(\omega)$ subject to Dirichlet boundary conditions is also stated (see Theorem \ref{thm:bc}). 

Then, Section \ref{sec:existence} is entirely devoted to the proof of the existence of the homogenized integrand $f_{\rm hom}$ and of its main properties (see Lemma \ref{l.existence_f_hom}). 

Section \ref{s.proofs} contains the proof of Theorem \ref{thm.Gamma_pure}. This proof is based on the preliminary technical results Lemma \ref{l.truncation} and Lemma \ref{l.fundamental_estimate} and then achieved in two main steps carried out in Proposition \ref{p.ub} and Proposition \ref{p.lb}. 

Eventually, in Section \ref{s.bc} the case of Dirichlet boundary conditions is considered and Theorem \ref{thm:bc} is proven, while in the short Appendix some measurability issues are addressed (see Lemma \ref{l.oninf}).

\section{Notation and preliminaries}\label{s.preliminaries}

\subsection{General notation}
Throughout the paper $d,m \in \N$ are fixed with $d,m\geq 2$. Given a measurable set $A\subset\R^d$, $|A|$ denotes its $d$-dimensional Lebesgue measure. The Euclidean norm of $x\in \R^d$ is denoted by $|x|$ and $B_{\rho}(x_0):=\{x\in \R^d \colon |x-x_0|<\rho\}$ denotes the open ball with radius $\rho>0$ centred at $x_0$.
Given $x_0\in\R^d$ and $\rho>0$ we set $Q_{\rho}(x_0):=x_0+(-\rho/2,\rho/2)^d$.  

For $\xi\in \R^{m\times d}$ fixed, $\ell_{\xi}$ denotes the linear function with gradient $\xi$, that is $\ell_{\xi}(x):=\xi x$.  

We define $\mathbb{D}_d$ to be the set of diagonal matrices in $\R^{d\times d}$. For a measurable set $A$ with positive measure, we define $\dashint_A:=\frac{1}{|A|}\int_A$. We use standard notation for $L^p$-spaces and Sobolev spaces $W^{1,p}$. The Borel $\sigma$-algebra on $\R^d$ is denoted by $\mathcal{B}^d$, while we use $\mathcal{L}^d$ for the $\sigma$-algebra of Lebesgue-measurable sets.

Throughout the paper, the parameter $\e>0$ varies in a strictly decreasing sequence of positive real numbers converging to zero. 

The letter $C$ stands for a generic positive constant which may vary from line to line, within the same expression.

\subsection{BV-functions}
In this section we recall some basic facts and notation concerning the space of functions of bounded variation. For a systematic treatment of this subject we refer the reader to the monograph \cite{AmbFusPal}.

Let $A\subset\R^d$ be an open set. A function $u\in L^1(A,\R^m)$ is a function of bounded variation if its distributional derivative $\D u$ is a finite matrix-valued Radon measure on $A$; in this case we write $u\in BV(A,\R^m)$.	
The space $BV(A,\R^m)$ is a Banach space when endowed with the norm $\|u\|_{BV(A,\R^m)}:=\|u\|_{L^1(A,\R^m)}+|\D u|(A)$, where $|\D u|$ denotes the total variation measure of $\D u$. If $A$ is a bounded Lipschitz domain, then $BV(A,\R^m)$ is compactly embedded in $L^q(A,\R^m)$ for $q<1^*:=d/(d-1)$. We say that a sequence $(u_k)$ converges weakly$^*$ in $BV(A,\R^m)$ to $u$ if $u_k\to u$ in $L^1(A,\R^m)$ and $\D u_n\overset{\star}{\rightharpoonup}\D u$ in the sense of measures.

If $u \in BV(A,\R^m)$ the structure of $\D u$ can be characterised. To this end, we need to recall some further concepts and notation.  
A function $u\in L^1(A,\R^m)$ has an approximate limit at $x\in A$ whenever there exists $z\in\mathbb{R}^d$ such that
\begin{equation*}
	\lim_{\rho\to 0^+}\frac{1}{\rho^d}\int_{B_{\rho}(x)}|u(y)-z|\,\mathrm{d}y=0\,.
\end{equation*}
Next we introduce the so-called approximate jump points of $u$. Given $x\in A$ and $n\in \mathbb S^{d-1}$ we set
\begin{equation*}
	B^{\pm}_{\rho}(x,n):=\{y\in B_{\rho}(x):\;\pm (y-x) \cdot n >0\}\,.
\end{equation*}
We say that $x\in A$ is an approximate jump-point of $u$ if there exist $a\neq b\in\mathbb{R}^m$ and $n\in \mathbb S^{d-1}$ such that
\begin{equation*}
	\lim_{\rho\to 0^+}\frac{1}{\rho^d}\int_{B_{\rho}^+(x,n)}|u(y)-a|\,\mathrm{d}y=\lim_{\rho\to 0^+}\frac{1}{\rho^d}\int_{B^-_{\rho}(x,n)}|u(y)-b|\,\mathrm{d}y=0 \, .
\end{equation*}
The triplet $(a,b,n)$ is unique up to the change to $(b,a,-n)$ and is denoted by $(u^+(x),u^-(x),n_u(x))$, moreover we let $J_u$ be the set of approximate jump-points of $u$. The triplet $(u^+,u^-,n_u)$ can be chosen as a Borel function on the Borel set $J_u$. Then, denoting by $\nabla u$ the approximate gradient of $u$, we can decompose the measure $\D u$ as 
\begin{equation*}
	\D u(B)=\int_B\nabla u\,\mathrm{d}x+\int_{J_u\cap B}(u^+ -u^-)\otimes n_u \,\mathrm{d}\mathcal{H}^{d-1}+\D^{(c)}u(B) \, ,
\end{equation*}
for every $B\in \mathcal B^d$, where $\D^{(c)}u$ is the Cantor part of $\D u$ and $\D^{(j)}u := (u^+ - u^- )\otimes n_u \mathcal{H}^{d-1} \LL J_u$ is its jump-part, with $\mathcal H^{d-1}$ being the $(d-1)$-dimensional Hausdorff measure. The total variation $|\D u|$ can then be decomposed as
\begin{equation*}
|\D u|(B)=\int_B |\nabla u|\dx+\int_{J_u\cap B}|u^+-u^-|\,\mathrm{d}\mathcal{H}^{d-1}+\int_B \frac{\mathrm{d} \D^{(c)} u}{\mathrm{d} |D^{(c)}u|}	\,\mathrm{d} |D^{(c)}u|.
\end{equation*}
In this paper we mostly use the simpler decomposition $\D u(B)=\int_B\nabla u(x)\dx+\D^s u(B)$, where $\D^s u$ denotes the part of the measure $\D u$ which is singular with respect to the Lebesgue measure, thus $\D^s u=\D^{(j)}u+\D^{(c)}u$.

\subsection{Ergodic theory}
Let $(\Omega,\mathcal{F},\mathbb{P})$ be a complete probability space. Below we recall some basic definitions and some useful results from ergodic theory.
\begin{definition}[Measure-preserving group-action]\label{def:group-action} A measure-preserving group-action on $(\Omega,\F,\P)$ is a family $\tau:=\{\tau_z\}_{z\in\R^d}$ of measurable mappings $\tau_z:\Omega\to\Omega$ satisfying the following properties:
\begin{enumerate}[label=(\arabic*)]
	\item\label{joint} (joint measurability) the map $(\w,z)\mapsto\tau_z(\w)$ is $(\mathcal{F}\otimes\mathcal{L}^d, \mathcal{F})$-measurable for every $z\in \R^d$;
	\item\label{inv} (invariance) $\P(\tau_z E)=\P(E)$, for every $E\in\F$ and every $z\in\R^d$;
	\item\label{group} (group property) $\tau_0=\rm id_\Omega$ and $\tau_{z_1+z_2}=\tau_{z_2}\circ\tau_{z_1}$ for every $z_1,z_2\in\R^d$.
\end{enumerate}
If, in addition, $\{\tau_z\}_{z\in\R^d}$ satisfies the implication
\begin{equation*}
	\mathbb{P}(\tau_z E\Delta E)=0\quad\forall\, z\in\R^d\implies \mathbb{P}(E)\in\{0,1\},
\end{equation*}
then $\tau$ it is called ergodic.
\end{definition}
Throughout the paper we frequently use a variant of the Birkhoff Ergodic Theorem which is useful for our purposes. Before stating it, we need to fix some additional notation. 

Let $g$ be a measurable function on $(\Omega,\mathcal{F},\mathbb P)$; $\mathbb{E}[g]$ denotes the expected value of $g$, that is
\begin{align*}
	\mathbb{E}[g]:=\int_\Omega g(\omega){\rm d}\mathbb P.
\end{align*}
For every $g\in L^1(\Omega)$ and for every $\sigma$-algebra $\mathcal{F}'\subset\mathcal{F}$, we denote with $\mathbb{E}[g|\mathcal{F}']$ the conditional expectation of $g$ with respect to $\mathcal{F}'$.
We recall that $\mathbb{E}[g|\mathcal{F}']$ is the unique $L^1(\Omega)$-function satisfying 
\begin{align*}
	\int_E\mathbb{E}[g|\mathcal{F}'](\omega){\rm d}\mathbb P=\int_E g(\omega){\rm d}\mathbb P
\end{align*}
for every $E\in\mathcal{F}'$.

\smallskip

We recall the following version of the Additive Ergodic Theorem which can be found in \cite[Lemma 4.1]{RR21}.

\begin{theorem}[Additive ergodic theorem]\label{thm.additiv_ergodic}
Let $g\in L^1(\Omega)$, let $\tau$ be a measure-preserving group-action on $(\Omega,\mathcal{F},P)$, and let $\mathcal{F}_\tau$ denote the $\sigma$-algebra of $\tau$-invariant sets. Then there exists a set $\Omega'\in\mathcal{F}$ with $P(\Omega')=1$ such that for every $\omega\in\Omega'$ and for every measurable bounded set $B\subset\mathbb{R}^d$ with $|B|>0$ there holds 
\begin{equation}\label{Birkhoff}
\lim_{t\to +\infty}\dashint_{tB}g(\tau_z\w)\,{\rm d}z=\mathbb{E}[g|\mathcal{F}_\tau](\omega).
\end{equation} 
If moreover $\tau$ is ergodic, then  $\mathcal{F}_\tau$ reduces to the trivial $\sigma$-algebra, therefore \eqref{Birkhoff} becomes
\begin{equation*}
\lim_{t\to +\infty}\dashint_{tB}g(\tau_{z}\omega)\,{\rm d}z=\mathbb{E}[g].
\end{equation*} 
\end{theorem}
For later purposes we also need to recall the definition of subadditive process. 

In all that follows $\mathcal A$ denotes the collection of all open and bounded subsets of $\R^d$ with Lipschitz boundary. 
\begin{definition}[Subadditive process]\label{def:sub-p}
	Let $\tau$ be a measure-preserving group-action on $(\Omega,\mathcal{F},P)$. A subadditive process is a function $\mu:\Omega\times\mathcal{A}\rightarrow [0,+\infty)$ satisfying the following properties:
	\begin{enumerate}
		\item for every $A\in\mathcal{A}$, $\mu(\cdot,A)$ is $\mathcal{F}$-measurable;
		\item for every $\omega\in\Omega$, $A\in\mathcal{A}$, and $z\in\mathbb{R}^d$
		\begin{align*}
			\mu(\omega,A+z)=\mu(\tau_z\omega,A);
		\end{align*}
		\item for every $\omega\in\Omega$, for every $A\in\mathcal{A}$, and for every \emph{finite} family $(A_i)_{i\in I}\subset\mathcal{A}$ of pairwise disjoint sets such that $A_i\subset A$ for every $i\in I$ and $|A\setminus\cup_{i\in I}A_i|=0$, there holds
		\begin{align*}
		\mu(\omega,A)\leq\sum_{i\in I}\mu(\omega,A_i);
		\end{align*}
	    	    \item there exists a constant $c>0$ such that for every $A\in\mathcal{A}$
	    \begin{align*}
	    	0\leq\mathbb E\big[\mu(\cdot,A)\big]\leq c|A|.
	    \end{align*}
    \end{enumerate}
Moreover,  if $\tau$ is ergodic then $\mu$ is called a subadditive ergodic process.
\end{definition}

We now state a version of the Subadditive Theorem, originally proven by Akcoglu and Krengel \cite{AkKr}, which is suitable for our purposes (see \cite[Theorem 4.3]{LICHT-MICHAILLE}).

\begin{theorem}\label{thm:ergodic}
	Let $\mu:\Omega\times\mathcal{A}\rightarrow[0,+\infty)$ be a subadditive process. Then there exist a $\mathcal{F}$-measurable function $\phi:\Omega\rightarrow[0,+\infty)$ and a set $\Omega'\in\mathcal{F}$ with $\mathbb P(\Omega')=1$ such that
	\begin{align*}
		\lim_{t\to+\infty}\frac{\mu(\omega,tQ)}{|tQ|}=\phi(\omega),
	\end{align*}
for every $\omega\in\Omega'$ and for every cube $Q:=Q_{\rho}(x_0)$ in $\mathbb{R}^d$. 
\end{theorem}
\section{Setting of the problem and statement of the main result}\label{s.results}

In this section we introduce the family of random functionals we are going to study and state the main result of the paper.  

Below we define the class of admissible random integrands we consider throughout. 

\begin{assumption}[Admissible integrands]\label{a.1} 
The function $f:\Omega\times\R^d\times\R^{m\times d}\to [0,+\infty)$ is $(\mathcal{F}\otimes\mathcal{L}^d\otimes\mathcal{B}^{m\times d})$-measurable and satisfies the following assumptions: 
\begin{itemize}
	\item[(A1)](lower semicontinuity) For every $\w\in\Omega$ and every $x\in\R^d$ the function $\xi\mapsto f(\w,x,\xi)$ is lower semicontinuous.

\smallskip

	\item[(A2)](degenerate growth conditions)
	There exist $\alpha>0$ and $(\mathcal F \otimes \mathcal L^d)$-measurable functions $\Am:\Omega\times\R^d\to \mathbb{D}_{d}$, $\lambda:\Omega\times\R^d\to [0,+\infty)$ with $|\Am(\cdot,0)|,\lambda(\cdot,0)\in L^1(\Omega)$ and $|\Am(\cdot,0)^{-1}|\in L^{\infty}(\Omega)$, such that
	 for every $\w\in\Omega,x\in\R^d$, and every $\xi\in \R^{m\times d}$ there holds
	\begin{equation}\label{gc}
		\alpha \,|\xi \Am(\w,x)|\leq f(\w,x,\xi)\leq |\xi \Am(\w,x)|+\lambda(\w,x).
	\end{equation}

\smallskip

	\item[(A3)] (ergodicity) There exists a measure-preserving, ergodic group-action $\tau=\{\tau_{z}\}_{z\in\R^d}$ such that 
	\begin{equation*}
		\begin{split}
			f(\tau_z\w,x,\xi)&=f(\w,x+z,\xi),
			\\
			\Am(\tau_z\w,x)&=\Am(\w,x+z), 
			\\
			\lambda(\tau_z\w,x)&=\lambda(\w,x+z),
		\end{split}
	\end{equation*}
for every $z\in  \R^d$ and every $(\w,x,\xi)\in\Omega\times\R^d\times\R^{m\times d}$.
\end{itemize}
\end{assumption}

For $\e>0$ and $\omega \in \Omega$, we consider the integral functionals $F_{\e}(\omega) \colon L^1_{\rm loc}(\R^d,\R^m) \times \mathcal A \longrightarrow [0,+\infty]$ defined as
\begin{equation}\label{F-e}
F_{\e}(\w)(u,A):=
\begin{cases}
	\displaystyle \int_A f(\w,\tfrac{x}{\e},\nabla u)\dx &\mbox{if $u\in W^{1,1}(A,\R^m)$},
	\cr
	+\infty &\mbox{otherwise,}
\end{cases}
\end{equation}
with $f$ satisfying Assumption \ref{a.1}.

The following theorem establishes a homogenization result for the random functionals $F_\e$ and is the main result of this paper. 

\begin{theorem}[Stochastic homogenization]\label{thm.Gamma_pure}
Let $f$ satisfy Assumption \ref{a.1}; for every $\e>0$ and every $\omega \in \Omega$ let $F_\e(\omega)$ be as in \eqref{F-e}. 
Then, there exists $\widetilde \Omega\in \mathcal F$ with $\P(\widetilde\Omega)=1$ such that:
\begin{enumerate}
\item[i.] (Existence of the homogenized integrand) For every $\w\in \widetilde\Omega$, $x_0\in \R^d$, $\rho >0$, and every $\xi \in \R^{m\times d}$ the following limit exists, is spatially homogeneous, and deterministic
\begin{equation}\label{lim-f-hom}
\lim_{t\to +\infty}\frac{1}{|tQ_\rho(x_0)|}\inf\left\{\int_{tQ_\rho(x_0)}f(\w, x, \nabla u)\dx :\,u\in \ell_\xi +W^{1,1}_0(tQ_\rho(x_0),\R^m)\right\}.
\end{equation}
\item[ii.] (Properties of the homogenized integrand) For every $\xi \in \R^{m \times d}$ set
\begin{align}\label{f-hom}
f_{\rm hom}(\xi):=  \lim_{t\to +\infty}\frac{1}{t^d}\mathbb E \left[\inf\left\{\int_{Q_t(0)}f(\cdot, x, \nabla u)\dx :\,u\in \ell_\xi +W^{1,1}_0(Q_t(0),\R^m)\right\}\right];
\end{align}
then $f_{\rm hom}$ is continuous, quasiconvex, and for every $\xi \in \R^{m\times d}$ it satisfies the following standard linear growth conditions
\begin{equation*}
	\alpha\,c_0|\xi|\leq f_{\rm hom}(\xi)\leq C_0|\xi|+C_1,
\end{equation*}
with 
\begin{equation}\label{lim-const}
\text{$c_0:=\||\Lambda(\cdot,0)^{-1}|\|_{L^{\infty}(\Omega)}^{-1}$,\quad $C_0:=\sup_{\eta \in \R^{m\times d}, |\eta|=1}\mathbb{E}[|\eta \Am(\cdot,0)|]$,\quad and $C_1:=\mathbb{E}[\lambda(\cdot,0)]$.}
\end{equation} 
\item[iii.] (Almost sure $\Gamma$-convergence) For every $\w\in \widetilde \Omega$ and every $A\in \A$ the functionals 
$F_{\e}(\w)(\cdot,A)$ $\Gamma$-converge in $L^1_{\rm loc}(\R^d,\R^m)$ to $F_{\rm hom}(\cdot, A)$ with $F_{\rm hom} \colon L^1_{\rm loc}(\R^d;\R^m)\times \mathcal A \longrightarrow [0,+\infty]$ given by
\begin{equation}\label{F-hom}
F_{\rm hom}(u)= \begin{cases}
\displaystyle \int_A f_{\rm hom}(\nabla u)\dx+\int_A f_{\rm hom}^{\infty}\left(\frac{\mathrm{d}D^su}{\mathrm{d}|D^su|}\right)\,\mathrm{d}|D^su| & \text{if $u\in BV(A,\R^m)$},
\cr
+\infty & \text{otherwise,} 
\end{cases}
\end{equation}
where $f_{\rm hom}^{\infty}$ denotes the recession function of $f_{\rm hom}$\ie for every $\xi \in \R^{m \times d}$
\begin{equation*}
f_{\rm hom}^{\infty}(\xi):=\limsup_{t\to +\infty}\frac{1}{t}f_{\rm hom}(t\xi).
\end{equation*}
\end{enumerate}
\end{theorem}

The two following remarks are in order. 

\begin{remark}[Assumptions on $f$]\label{r.on_assump}
Below we comment on the requirements on the integrand $f$ as in Assumption \ref{a.1}. 
\begin{itemize}
	\item[a)] The continuity property in (A1) is needed to prove the existence of the limit in \eqref{lim-f-hom} via Theorem \ref{thm:ergodic}. Specifically, (A1) is only needed to prove the measurability of the process. 
	\item[b)] The growth condition in \eqref{gc} can be replaced by the weaker condition
	\begin{equation*}
		\alpha |\xi \Am(\w,x)|-\lambda(\w,x)\leq f(\w,x,\xi)\leq |\xi \Am(\w,x)|+\lambda(\w,x),
	\end{equation*}
	for every $\omega \in \Omega$, $x\in \R^d$, and $\xi \in \R^{m \times d}$.
	Indeed the integrand $\tilde{f}(\w,x,\xi):=f(\w,x,\xi)+\lambda(\w,x)$ satisfies Assumption \ref{a.1} with $\lambda$ replaced by $2\lambda$. Therefore for such an integrand Theorem \ref{thm.Gamma_pure} gives the same $\Gamma$-limit as in \eqref{F-hom} up to the additive constant $\mathbb{E}[\lambda(\cdot,0)]|A|$.
	\item[c)] The integrability assumptions on $\lambda (\cdot,0)$ and $\Lambda (\cdot, 0)$ together with (A3) and Fubini's Theorem imply that for $\mathbb P$-a.e. $\w\in\Omega$ there holds: $|\Am(\w,\cdot)|,\lambda(\w,\cdot)\in L^1_{\rm loc}(\R^d)$. Furthermore the $L^\infty(\Omega)$-bound on $|\Lambda (\cdot,0)|^{-1}$ combined with \cite[Lemma 7.1]{JKO} gives that, for $\mathbb P$-a.e. $\w\in\Omega$, $|\Am(\w,\cdot)^{-1}|\in L^{\infty}(\R^d)$ with a bound uniform with respect to $\w$. Therefore, in particular, for $\mathbb P$-a.e. $\w\in\Omega$ the following holds true: there exists a (deterministic) constant $C>0$ such that for a.e. $x\in\R^d$ and for every $\xi\in\R^{m\times d}$ we have
	\begin{equation}\label{eq:linear_from_below}
	|\xi|\leq |\xi \Am(\w,x)||\Am(\w,x)^{-1}|\leq C|\xi \Am(\w,x)|\leq \frac{C}{\alpha} f(\w,x,\xi). 	
	\end{equation}
\item[d)] In (A3) the ergodicity assumption on $\tau$ can be dropped. In fact if $f, \lambda$, and $\Lambda$ are only stationary, that is, they satisfy (A3) with respect to a measure-preserving group-action $\tau$ which is not necessarily ergodic, then an almost sure homogenization result for $F_\e(\omega)$ can still be established. However in this case the homogenization is not ``effective'' meaning that $f_{\rm hom}$ is still a random variable. Namely, we have
\[
f_{\rm hom}(\omega,\xi):=  \lim_{t\to +\infty}\frac{1}{t^d}\inf\left\{\int_{Q_t(0)}f(\w, x, \nabla u)\dx :\,u\in \ell_\xi +W^{1,1}_0(Q_t(0),\R^m)\right\};
\] 	
moreover, in this case $f_{\rm hom}$ satisfies the following bounds
\begin{equation*}
	\alpha\,c_0|\xi|\leq f_{\rm hom}(\xi)\leq C_0(\w)|\xi|+C_1(\w),
\end{equation*}
with $c_0$ as in \eqref{lim-const} and
\begin{equation*}
C_0(\w):=\hspace{-0.2cm}\sup_{\eta \in \R^{m\times d}, |\eta|=1}\hspace{-0.2cm}\mathbb{E}[|\eta \Am(\cdot,0)||\mathcal{F}_\tau],\quad  C_1(\w):=\mathbb{E}[\lambda(\cdot,0)|\mathcal{F}_\tau].
\end{equation*}
\end{itemize}
\end{remark} 

\begin{remark}[Optimality of the assumptions on $\Am(\cdot,0)$]\label{r.optimal}
	Both the assumption $|\Am(\cdot,0)| \in L^1(\Omega)$ and $|\Am(\cdot,0)^{-1}| \in L^\infty(\Omega)$ are optimal in the sense that if we drop one of the two we can either get $f_{\rm hom}(\xi)= +\infty$ on a subspace of $\R^{m\times d}$ or a loss of compactness in $BV$, meaning that interfaces can be approximated at zero cost. 
	
	Both these effects can be shown by adapting an example in \cite[Remark 3.7]{NSS} of a discrete laminate-like structure to our setting. For the sake of the exposition we treat here the case $m=1$ and consider a sequence $(a_k)_{k\in\Z}:\Omega\to (0,+\infty)$ of i.i.d. random variables. For every $x\in \R^d$ we define the piecewise constant interpolation corresponding to $(a_k)_{k\in\Z}$ as
	\[
	a(\omega,x):=a_k(\omega)\quad \text{if \;$x\in [k,k+1)$,\; $k\in\Z$}.  
	\]
	On the product space $\Omega^{\Z}$ one can define a stationary, ergodic group-action which turns $a$ into a $\Z$-stationary, ergodic function. (If one is interested in an $\R^d$-stationary, ergodic example, one can turn $a$ into a stationary and ergodic function with respect to all translations on the extended probability space $\mathbb{T}\times\Omega$, where $\mathbb T$ is the torus in $\R^d$, by setting $a((z,\w),x):=a(\w,z+x)$, thus preserving the piecewise-constant structure. See \cite[p. 236]{JKO} for more details.) 
	
	For $x:=(x_1,x')\in \R^d$ define $f(\w,x,\xi)=|a(\w,x_1)\xi|$, so that $\Am(\w,x)=a(\w,x_1)I_d$, where $I_d$ is the identity matrix in $\R^{d\times d}$. Assume that $\mathbb{E}[a(\cdot,0)]=+\infty$ so that  $|\Am(\cdot,0)| \notin L^1(\Omega)$ and for $k\in\N$ let $u\in W^{1,1}_0(kQ)$, with $Q:=(0,1)^d$. Defining the lower dimensional cube $Q':=(0,1)^{d-1}$, for a.e. $x_1\in (0,k)$ it holds that $u(x_1,\cdot)\in W_0^{1,1}(kQ')$ and 
	\begin{equation*}
		\dashint_{kQ}|a(\w,x_1)(\nabla u(x)+\xi)|\dx\geq\dashint_0^ka(\w,x_1) \dashint_{kQ'}|\nabla_{x'}u(x_1,x')+(\xi_2, \dots, \xi_d)|{\rm d}x'\,\mathrm{d}x_1.
	\end{equation*}
	The inner integral on the right-hand side is minimal for $u(x_1,\cdot)\equiv 0$ due to Jensen's Inequality. Hence
	\begin{equation*}
		\inf\left\{\dashint_{kQ}|a(\w,x)(\nabla u(x)+\xi)|\dx:\,u\in W^{1,1}_0(kQ)\right\}\geq |(\xi_2, \dots, \xi_d)|\dashint_{0}^ka(\w,x_1)\dx_1.
	\end{equation*}
	Combining a truncation of the weight $a$ with the Ergodic Theorem \ref{thm.additiv_ergodic}, for $\xi\notin\R e_1$ it follows that a.s.
	\begin{equation}\label{eq:toinfinity}
		f_{\rm hom}(\xi):=\lim_{k\to +\infty}\inf\left\{\dashint_{kQ}|a(\w,x)(\nabla u(x)+\xi)|\dx:\,u\in W^{1,1}_0(kQ)\right\}=+\infty.
	\end{equation}
	Next we consider the case when $a(\cdot,0)^{-1}\notin L^{\infty}(\Omega)$. Then for every $\delta>0$ there exists a set $\Omega_{\delta}$ with positive probability such that $|a_0(\w)|<\delta$ for all $\w\in\Omega_{\delta}$. In particular, by stationarity this implies that $\mathbb{P}(|a_k|<\delta)=p_{\delta}>0$ for all $k\in\N$. Moreover, the independence yields 
	\begin{equation*}
	\sum_{n=0}^{+\infty}\mathbb{P}\left(\min_{0\leq k\leq n}|a_k|\geq \delta\right)=\sum_{n=0}^{+\infty}\mathbb{P}\left(|a_k|\geq \delta\;\text{ for all }0\leq k\leq n\right)=\sum_{n=0}^{+\infty}(1-p_{\delta})^{n+1}<+\infty.
	\end{equation*}
	Hence the Borel-Cantelli Lemma implies that a.s. there exists $k_\delta=k_{\delta}(\w)$ such that $|a_{k_\delta}(\w)|<\delta$. For $\e>0$ we define the piecewise affine function $u_{\e}(x):=\max\{0,\min\{1,\e^{-1}x_1-k_{\delta}\}\}$ (this sequence has to be slightly shifted if we work on the extended probability space). Clearly, $u_{\e}(x)=0$ for $x_1\leq\e k_{\delta}$ and $u_{\e}(x)=1$ for $x_1> (k_{\delta}+1)\e$. Hence $u_{\e}\to \chi_{\{x_1<0\}}$ in $L^1_{\rm loc}(\R^d)$ and $\nabla u_\e$ is concentrated on the stripe $\{x\in \R^d \colon k_{\delta}\e\leq x_1\leq (k_{\delta}+1)\e\}$, where the weight-function is smaller than $\delta$ by construction. Therefore we have 
	\begin{equation*}
	F_{\e}(\w)(u_{\e},Q)\leq \int_{\e k_{\delta}}^{\e (k_{\delta}+1)}\delta \e^{-1}\,\mathrm{d}x_1=\delta,
	\end{equation*}
	thus by the arbitrariness of $\delta>0$ we deduce that
	\begin{equation*}
		\Gamma\hbox{-}\lim_{\e\to 0^+}F_{\e}(\w)(\chi_{\{x_1<0\}},Q)=0,
	\end{equation*}
	and hence the claim.
	
	With some more effort one can prove the same result for any plane interface of the form $\chi_{\{x_1<r\}}$ with $r\in\Q$ by applying the same Borel-Cantelli-type argument to the events
	\begin{equation*}
		\left\{\min_{0\leq k\leq n}|a_{\lceil\tfrac{r}{\e}\rceil+k}|\geq\delta\right\},
	\end{equation*}
	whose probabilities do not depend on $r$ thanks to stationarity.
	Then by lower semicontinuity the same holds for all $r\in\R$. Since the constructions are local, eventually one can then approximate a $BV$-function with arbitrarily large $BV$-norm paying zero energy. 
\end{remark}

\subsection{Dirichlet boundary conditions} 
In this subsection we state a homogenization result for the functionals $F_\e(\w)$ subject to Dirichlet boundary conditions. To this end, we need to define a class of admissible boundary data. Since the weight function $\Lambda(\omega,\cdot)$ only belongs to $L^1_{\rm loc}(\R^d)$, we need to consider sufficiently regular boundary data.

\begin{assumption}[Admissible boundary data]\label{a.2}
The function $u_0$ belongs to $W^{1,1}(\R^d,\R^m)$. Moreover there exists $\widehat \Omega \in \F$ with $\mathbb P(\widehat \Omega)=1$ such that for every $\omega \in \widehat \Omega$ the sequence of functions $(M_\e(\w))_\e$ defined as
\begin{equation}\label{eq:boundarymeasure}
M_{\e}(\w)(x):=|\nabla u_0(x)\Lambda(\w,\tfrac{x}{\e})|+|u_0(x)||\Am(\w,\tfrac{x}{\e})|
\end{equation}
is locally equi-integrable on $\R^d$.
\end{assumption}
We notice that in view of Theorem \ref{thm.additiv_ergodic} Lipschitz-functions with compact support satisfy \eqref{eq:boundarymeasure} and hence are admissible boundary data. However, fixing a bounded, open set $A$ with Lipschitz boundary, this restricts the boundary value $u_0|_{\partial A}$ also to Lipschitz-functions. 

Since $|\Lambda(\omega,\cdot)|$ only belongs to $L^1_{\rm loc}(\R^d)$, the request in \eqref{eq:boundarymeasure} is necessary in order to have at least one competitor for the minimization problem with Dirichlet boundary condition $u_0$. 

On the other hand, \eqref{eq:boundarymeasure} can be relaxed when $|\Lambda(\cdot,0)|$ has higher stochastic moments. In fact, H\"older's Inequality ensures that $M_{\e}(\w)$ is always locally equi-integrable on $\R^d$ when $u_0\in W^{1,p}(\R^d,\R^m)$ and $|\Lambda(\cdot,0)|\in L^{p/(p-1)}(\Omega)$ for some $p\in [1,+\infty)$.

Let $A\in \A$ and consider the functionals defined as
\begin{equation}\label{Fe-D}
F^{u_0}_{\e}(\w)(u,A):=\begin{cases}
\displaystyle\int_A f(\w,\tfrac{x}{\e},\nabla u)\dx &\mbox{if $u\in u_0+W_0^{1,1}(A,\R^m)$,}
\\
+\infty &\mbox{otherwise in $L^1(A,\R^m)$},
\end{cases}
\end{equation}
where $f$ satisfies Assumption \ref{a.1} and $u_0$ satisfies Assumption \ref{a.2}. 
The following $\Gamma$-convergence result holds true.
\begin{theorem}[Stochastic homogenization with Dirichlet boundary conditions]\label{thm:bc}
Let $f$ satisfy Assumption \ref{a.1} and let $u_0$ satisfy Assumption \ref{a.2}. Then, almost surely, the functionals $F^{u_0}_{\e}(\w)(\cdot,A)$ defined in \eqref{Fe-D} $\Gamma$-converge in $L^1(A,\R^m)$ to the deterministic functional $F^{u_0}_{\rm hom}(\cdot,A)$ given by
\begin{equation*}
F^{u_0}_{\rm hom}(u,A):=
\begin{cases}
\displaystyle\int_A f_{\rm hom}(\nabla u)\dx+\int_A f_{\rm hom}^{\infty}\left(\frac{\mathrm{d}D^su}{\mathrm{d}|D^su|}\right)\,\mathrm{d}|D^su|+\int_{\partial A}f_{\rm hom}^{\infty}\left((u_0^+-u^-)\otimes n_{\partial A}\right)\,\mathrm{d}\mathcal{H}^{d-1}\\ 
\hspace{8.6cm} \mbox{if $u\in BV(A,\R^m)$,}
\cr \cr
+\infty \hspace{8cm} \mbox{otherwise in $L^1(A,\R^m)$,}
\end{cases}
\end{equation*}
for every $A\in \A$. 
\end{theorem}

\smallskip

We observe that in the statement of Theorem \ref{thm:bc} the boundary integral
\[
\int_{\partial A}f_{\rm hom}^{\infty}\left((u_0^+-u^-)\otimes n_{\partial A}\right)\,\mathrm{d}\mathcal{H}^{d-1}
\]
represents the energy contribution of the function $u\chi_A+u_0(1-\chi_A)$ restricted to $\partial A$ and penalises the violation of the boundary constraint $u=u_0$ on $\partial A$, meant in the sense of traces.

\begin{remark}[Linear forcing terms]\label{r.bc}
Let $q > d$ and let $(g_\e) \subset L^q_{\rm loc}(\R^d,\R^m)$ be such that 
$g_{\e}\rightharpoonup g$ weakly in $L^q_{\rm loc}(\R^d,\R^m)$, for some $g\in L^q_{\rm loc}(\R^d,\R^m)$.
Consider the linear functionals
\[
G_\e(u,A):=\int_A g_{\e}u\dx \quad \text{and} \quad G(u,A):=\int_A g u\dx
\]
and the perturbed functionals $F_\e(\omega)(\cdot,A)+G_\e(\cdot,A)$. By assumption $G_\e$ can be regarded as a continuously converging perturbation, therefore it is immediate to check that Theorem \ref{thm.Gamma_pure} also yields the almost sure $\Gamma$-convergence of $F_\e(\omega)(\cdot,A)+G_\e(\cdot,A)$  to $F_{\rm hom}(\cdot,A)+G(\cdot,A)$, for every $A\in \A$. An analogous result can be proven for the functionals $F^{u_0}_\e(\omega)(\cdot,A)+G_\e(\cdot,A)$.
\end{remark}


\section{Existence of $f_{\rm hom}$}\label{sec:existence}

This section is devoted to proving the existence of the spatially homogeneous and deterministic integrand $f_{\rm hom}$. The proof of this result will be achieved by following a classical strategy introduced in \cite{DMMoII} and based on the Subadditive Ergodic Theorem \cite[Theorem 2.7]{AkKr}. 

\begin{lemma}[homogenization formula]\label{l.existence_f_hom}
Let $f$ satisfy Assumption \ref{a.1};  then, there exists $\Omega'\in \mathcal F$ with $\P(\Omega')=1$ such that for every $\w\in \Omega'$, $x_0\in \R^d$, $\rho >0$, and every $\xi \in \R^{m\times d}$ there exists
\begin{align}\nonumber
& \lim_{t\to +\infty}\frac{1}{|tQ_\rho(x_0)|} \inf\left\{\int_{tQ_\rho(x_0)}f(\w, x, \nabla u)\dx :\,u\in \ell_\xi +W^{1,1}_0(tQ_\rho(x_0),\R^m)\right\}
\\\label{c:det}
&=  \lim_{t\to +\infty}\frac{1}{t^d}\mathbb E \left[\inf\left\{\int_{Q_t(0)}f(\cdot, x, \nabla u)\dx :\,u\in \ell_\xi +W^{1,1}_0(Q_t(0),\R^m)\right\}\right]=:f_{\rm hom}(\xi).
\end{align}
The homogeneous and deterministic function $f_{\rm hom}$ is continuous and for every $\xi \in \R^{m\times d}$ it satisfies the following standard linear growth conditions
\begin{equation*}
	\alpha\,c_0|\xi|\leq f_{\rm hom}(\xi)\leq C_0|\xi|+C_1,
\end{equation*}
with 
\begin{equation*}
\text{$c_0:=\||\Lambda(\cdot,0)^{-1}|\|_{L^{\infty}(\Omega)}^{-1}$,\quad $C_0:=\sup_{\eta \in \R^{m\times d}, |\eta|=1}\mathbb{E}[|\eta \Am(\cdot,0)|]$,\quad and $C_1:=\mathbb{E}[\lambda(\cdot,0)]$.}
\end{equation*} 
\end{lemma}

\begin{proof}
Let $\xi\in \R^{m \times d}$ be fixed. For every $\omega\in \Omega$ and $A \in \A$ set

\begin{equation}\label{c:mu-xi}
		\mu_{\xi}(\w,A):=\inf\left\{\int_{A}f(\w, x, \nabla u)\dx :\,u\in \ell_\xi +W^{1,1}_0(A,\R^m)\right\}.
\end{equation}
We claim that $\mu_\xi$ is a subadditive process\ie $\mu_\xi$ satisfies properties $(1)\hbox{-}(4)$ in Definition \ref{def:sub-p}.  
	
	We start observing that for every $A\in \A$ the $\mathcal F$-measurability of $\mu_\xi(\cdot,A)$ follows by Lemma \ref{l.measurable} in the appendix.  Moreover, it is easy to check that for every $\omega\in\Omega$, $A\in\mathcal{A}$, and $z\in\mathbb{R}^d$
		\begin{align}\label{eq:mu_stationary}
			\mu_\xi(\omega,A+z)=\mu_\xi(\tau_z\omega,A).
		\end{align}
		Indeed, given $v\in \ell_\xi+W_0^{1,1}(A,\R^m)$, the function $\tilde{v}(x):=v(x-z)+\ell_\xi(z)$ belongs to $\ell_\xi +W_0^{1,1}(A+z,\R^m)$ and by (A3) and the translation invariance in $v$ we get
	\begin{equation*}
		\int_{A+z}f(\w,x,\nabla\tilde{v})\dx=\int_{A}f(\w,x+z,\nabla v)\dx=\int_{A}f(\tau_z\w,x,\nabla v)\dx,
	\end{equation*}
	then \eqref{eq:mu_stationary} follows by taking the inf in the above equality.
	
	 We now prove the subadditivity of $\mu_\xi$ as a set function. To this end let $\omega\in\Omega$ and $A\in\mathcal{A}$ be fixed and let $(A_i)_{i\in I}\subset\mathcal{A}$ be a finite family of pairwise disjoint sets such that $A_i\subset A$ for every $i\in I$ and $|A\setminus\cup_{i\in I}A_i|=0$. Let $\eta>0$; for
	every $i\in \mathcal I$ let $v_i\in \ell_\xi+W_0^{1,1}(A_i,\R^m)$ be such that
	\begin{equation}\label{inf-i}
	 \int_{A_i}f(\w, x, \nabla v_i)\dx\leq \mu_{\xi}(\w,A_i)+\eta.
	\end{equation}
	Define $v:=\sum_{i\in \mathcal I} v_i\chi_{A_i}$; clearly $v\in \ell_\xi+W^{1,1}_0(A,\R^m)$, therefore by additivity and locality we get
	\begin{equation*}
		\mu_{\xi}(\w,A)\leq \int_{A}f(\w, x, \nabla v)\dx =\sum_{i\in \mathcal I} \int_{A_i}f(\w, x, \nabla v_i)\dx,
	\end{equation*}
	thus the subadditivity is an immediate consequence of \eqref{inf-i}.    
	
	We are now left to prove the upper bound on $\mathbb E[\mu_\xi(\cdot, A)]$, the lower bound being trivial by the nonnegativity of $f$.

	Let $\omega\in\Omega$ and $A\in\mathcal{A}$ be fixed and arbitrary; choosing $u=\ell_\xi$ as a test function in the definition of $\mu_\xi(\omega,A)$, by \eqref{gc} we readily deduce that
	\begin{equation}\label{eq:pointwisebound}
		\mu_{\xi}(\w,A)\leq \int_A f(\w,x,\xi)\,\dx\leq \int_{A}\big(|\xi \Am(\w,x)|+\lambda(\w,x)\big)\,\mathrm{d}x.
	\end{equation}
	Therefore Tonelli's Theorem gives
	\begin{align}\nonumber
		\mathbb{E}\left[\mu_{\xi}(\cdot,A)\right]&\leq \int_A\mathbb{E}[|\xi \Am(\cdot,x)|]+\mathbb{E}[\lambda(\cdot,x)]\,\mathrm{d}x
		\\\nonumber
		&=\big(\mathbb{E}[|\xi \Am(\cdot,0)|]|+\mathbb{E}[\lambda(\cdot,0)]\big)|A|\nonumber
		\\\label{eq:mu_bound}
		&\leq \Big(\sup_{|\eta|= 1}\mathbb{E}[|\eta \Am(\cdot,0)|]|\xi|+\mathbb{E}[\lambda(\cdot,0)]\Big)|A|,
	\end{align}
	where the equality follows from a change of variables in $\Omega$ and the stationarity of $\Am$ and $\lambda$. 
	Eventually, for every fixed $\xi\in \R^{d\times m}$, $\mu_\xi$ satisfies the desired upper bound, thus is a subadditive process as claimed.
	
	Hence we can appeal to Theorem \ref{thm:ergodic} to deduce the existence of $\Omega_\xi \subset \Omega$ with $\Omega_\xi\in \F$, $\mathbb P(\Omega_\xi)=1$ and of a $\F$-measurable function $\phi_\xi$ such that for every $\omega \in \Omega_\xi$  there holds		
	\begin{equation}\label{lim-xi-fixed}
	\lim_{t\to +\infty} \frac{\mu_\xi(\omega, tQ_\rho(x_0))}{|tQ_\rho(x_0)|}=\lim_{t\to +\infty} \frac{\mu_\xi(\omega, Q_t(0))}{t^d}=:\phi_\xi(\omega), 
	\end{equation}
	for every $x_0 \in \R^d$ and every $\rho>0$. 
	
Now set 	
\[
\Omega' := \bigcap_{\xi \in \mathbb Q^{m\times d}} \Omega_\xi; 	
\]	
clearly $\mathbb P(\Omega')=1$ and in view of \eqref{lim-xi-fixed}, for every $\omega \in \Omega'$ the function $\phi_\xi$ is well defined for $\xi \in \mathbb Q^{m\times d}$. 

We now claim that for every  $\omega \in \Omega'$ the limit in \eqref{lim-xi-fixed} exists for every $\xi \in \R^{m\times d}$. We introduce the auxiliary functions $\phi^+_\rho, \phi^-_\rho \colon \Omega' \times \R^d \times \R^{m\times d} \to [0,+\infty)$ defined as
\begin{equation*}
\phi_\rho^+(\omega, x, \xi):= \limsup_{t \to +\infty}\frac{\mu_\xi(\omega, Q_{\rho t}(tx))}{\rho^d t^d}, \qquad \phi_\rho^-(\omega, x, \xi):= \liminf_{t \to +\infty}\frac{\mu_\xi(\omega, Q_{\rho t}(tx))}{\rho^d t^d}. 
\end{equation*}
By definition we have that 
\begin{equation}\label{c:first}
\phi^+_{\rho}(\omega, x, \xi)=\phi^-_\rho(\omega, x, \xi)=\phi_\xi(\omega), 
\end{equation}
for every $\omega \in \Omega'$, $x\in \R^d$,  $\xi \in \mathbb Q^{m\times d}$, and $\rho>0$. 

Let $\delta\in (0,1)$ be fixed; we then have 
\[
Q_{\rho(1-\delta)t}(tx) \subset \subset  Q_{\rho t}(tx) \subset \subset  Q_{\rho (1+\delta)t}(tx). 
\]
Moreover let $\xi\in\R^{m\times d}$ and $(\xi_j) \subset \Q^{m \times d}$ be such that $\xi_j \to \xi$, as $j \to +\infty$.
Consider $v\in\ell_\xi+W^{1,1}_0(Q_{\rho t}(tx),\R^m)$ arbitrary and extend it to $\R^d$ by setting $v=\ell_\xi$ in $\R^d\setminus Q_{\rho t}(tx)$. 
Let $\varphi\in C_c^{\infty}(\R^d,[0,1])$ be a cut-off function such that 
	\begin{equation*}
		\varphi\equiv 1\,\text{ on }\,Q_{\rho t}(tx),\quad\quad\varphi\equiv 0\,\text{ on }\,\R^d\setminus Q_{\rho(1+\delta) t}(tx),\quad\quad\|\nabla\varphi\|_{L^{\infty}(\R^d)}\leq \frac{2}{\rho \delta t}.
	\end{equation*}	
Define $\tilde{v}:=\varphi v+(1-\varphi)\ell_{\xi_j}$; clearly, $\tilde{v}\in \ell_{\xi_j}+W_0^{1,1}(Q_{\rho (1+\delta)t}(tx),\R^m)$.

	By definition of $\tilde v$ and \eqref{gc}, setting $\kappa:=|\Am| +\lambda$ we have 
	\begin{align*}
		\mu_{\xi_j}(\w, Q_{\rho (1+\delta)t}(tx))&\leq \int_{Q_{\rho (1+\delta)t}(tx)}f(\w,y,\nabla \tilde v)\dy
		\\
		&\leq\int_{Q_{\rho t}(tx)}f(\w,y,\nabla v)\dy+\int_{Q_{\rho (1+\delta)t}(tx) \setminus Q_{\rho t}(tx)}\hspace{-0.7cm}\big(|\nabla\tilde{v}||\Am(\w,y)|+\lambda(\w,y)\big)\dy
		\\
		&\leq \int_{Q_{\rho t}(tx)}f(\w,x,\nabla v)\dx+C\int_{Q_{\rho (1+\delta)t}(tx) \setminus Q_{\rho t}(tx)}\hspace{-1cm}\kappa(\w,y)(|\nabla\varphi||\xi -\xi_j| |y|+|\xi|+|\xi_j|+1)\dy
		\\
		&\leq \int_{Q_{\rho t}(tx)}f(\w,y,\nabla v)\dy+C\int_{Q_{\rho (1+\delta)t}(tx) \setminus Q_{\rho t}(tx)}\hspace{-1cm}\kappa(\w,y)\Big(|\xi-\xi_j|\frac{|y|}{\rho \delta t}+|\xi|+|\xi_j|+1\Big)\dy.
	\end{align*}
	Since $\delta< 1$,  we have that $|y| \leq \sqrt d (\rho + |x|)t$ in $Q_{\rho (1+\delta)t}(tx)$.  Then, we can pass to the inf over $v$ to deduce that
	\begin{equation*}
		\mu_{\xi_j}(\w,Q_{\rho (1+\delta)t}(tx))\leq  \mu_{\xi}(\w,Q_{\rho t}(tx ))+C\bigg(\frac{\rho +|x|}{\rho\delta}|\xi-\xi_j|+|\xi|+|\xi_j|+1\bigg)\int_{Q_{\rho (1+\delta)t}(tx) \setminus Q_{\rho t}(tx)}\hspace{-0.5cm}\kappa(\w,y)\dy.
	\end{equation*}
Then, dividing both sides of the expression above by $(\rho t)^d$, passing to the liminf as $t \to +\infty$, and invoking Theorem \ref{thm.additiv_ergodic} we get 	
\begin{equation}\label{c:uno}
(1+\delta)^d\phi^-_\rho\Big(\omega, \frac{x}{1+\delta},\xi_j\Big) \leq \phi^-_\rho(\omega, x,\xi)+ C\bigg(\frac{\rho +|x|}{\rho \delta}|\xi-\xi_j|+|\xi|+|\xi_j|+1\bigg) \left({(1+\delta)^d}-1\right)\mathbb{E}[\kappa(\cdot,0)]
\end{equation}	
(we notice here that in principle Theorem \ref{thm.additiv_ergodic} holds in a set of probability one, say $\Omega''$, which can be smaller than $\Omega'$.  However since clearly $\mathbb P(\Omega'\cap \Omega'')=1$ with a little abuse of notation we still denote with $\Omega'$ this intersection). 
Analogously we can prove that 
\begin{equation}\label{c:due}
\phi^+_\rho(\omega, x,\xi) \leq (1-\delta)^d\phi^+_\rho\Big(\omega, \frac{x}{1-\delta},\xi_j\Big) +C\bigg(\frac{\rho +|x|}{\rho\delta}|\xi-\xi_j|+|\xi|+|\xi_j|+1\bigg) \left({1-(1-\delta)^d}\right)\mathbb{E}[\kappa(\cdot,0)].
\end{equation}
Hence since $(\xi_j) \subset \mathbb Q^{m\times d}$, thanks to \eqref{c:first} we have 
\begin{equation}\label{c:tre}
\phi^-_\rho\Big(\omega, \frac{x}{1+\delta},\xi_j\Big)= \phi^+_\rho\Big(\omega, \frac{x}{1-\delta},\xi_j\Big)=\phi_{\xi_j}(\omega),
\end{equation}
for every $\omega \in \Omega'$, $x\in \R^d$, $j\in \mathbb N$, and $\rho>0$. 
Therefore, gathering \eqref{c:uno}-\eqref{c:tre}, passing to the liminf as $j \to +\infty$ and as $\delta \to 0^+$ we get 
\[
\liminf_{j \to +\infty} \phi_{\xi_j}(\omega) \leq \phi^+_\rho(\omega, x,\xi) \leq \phi^-_\rho(\omega, x,\xi)  \leq \liminf_{j \to +\infty} \phi_{\xi_j}(\omega),
\]
hence
\[
\phi^+_\rho(\omega, x,\xi) = \phi^-_\rho(\omega, x,\xi) =\liminf_{j \to +\infty} \phi_{\xi_j}(\omega),
\]
for every $\omega \in \Omega'$, $x\in \R^d$, $\xi \in \R^{m\times d}$, and $\rho>0$. Note that in particular all the terms in the above equalities do not depend on $x$ and $\rho$. Then, by the definition of $\phi^\pm_\rho$ we have that for every $\omega \in \Omega'$ and every $\xi \in \R^{m\times d}$ the following limit exists and does not depend on $x$ and on $\rho>0$
\[
\lim_{t\to +\infty} \frac{\mu_\xi(\omega, tQ_\rho(x))}{|tQ_\rho(x)|}.
\]
We then set
\begin{equation}\label{c:def}
\phi_\xi(\omega):=\phi^+_\rho(\omega, x,\xi) = \phi^-_\rho(\omega, x,\xi)= \lim_{t\to +\infty} \frac{\mu_\xi(\omega, tQ_\rho(x))}{|tQ_\rho(x)|}=\lim_{t\to +\infty} \frac{\mu_\xi(\omega, Q_t(0))}{t^d}. 	
\end{equation}
Then, for every $\xi\in \R^{m\times d}$ the function $\omega \mapsto \phi_\xi(\w)$ is $\F$-measurable on $\Omega'$ while by \eqref{c:uno}, \eqref{c:due}, and \eqref{c:def} the function $\xi \mapsto \phi_\xi(\w)$ is continuous for every $\omega\in \Omega'$. 
	
We now prove that $\phi_\xi$ is actually deterministic. By the ergodicity of $\{\tau_z\}_{z\in \R^d}$, this is equivalent to proving that 
\begin{equation}\label{c:invariance}
\phi_\xi(\omega)=\phi_\xi(\tau_z\omega),
\end{equation}
for every $\xi \in \R^{m\times d}$ and $z\in \R^d$ (cf. \cite[Corollary 6.3]{CDMSZ19}). 	
	Let $\omega\in \Omega$, $\xi\in \R^{m\times d}$, and $z\in \R^d$ be fixed. By stationarity we have 
	\begin{equation}\label{c:stat}
	\mu_{\xi}(\tau_z\w,tQ)=\mu_{\xi}(\w,t(Q+z/t)),
	\end{equation}
	for every $t>0$, where $Q:=Q_{\rho}(x_0)$. Given $Q'$ and $Q''$ open cubes with $Q'\subset\subset Q\subset\subset Q''$, for $t>0$ large enough we have 
	\begin{equation*}
		Q' \subset Q+z/t \subset Q''.
	\end{equation*}
	By the subadditivity of $\mu_{\xi}$ and \eqref{eq:pointwisebound} we get
	\begin{align*}
		\frac{\mu_{\xi}(\w,t(Q+z/t))}{|tQ|}&\leq \frac{\mu_{\xi}(\w,tQ')}{|tQ'|}+\frac{1}{|tQ|}\int_{t(Q+z/t)\setminus tQ'}\big(|\xi \Am(\w,x)|+\lambda(\w,x)\big)\,\mathrm{d}x
		\\
		&\leq \frac{\mu_{\xi}(\w,tQ')}{|tQ'|}+\frac{(|\xi|+1)}{|tQ|}\int_{tQ''\setminus tQ'}\kappa(\w,x)\dx.
	\end{align*}
	In view of \eqref{c:stat}, passing to the limsup as $t\to +\infty$, and invoking Theorem \ref{thm.additiv_ergodic} we infer that
	\begin{equation*}
		\limsup_{t\to +\infty}\frac{\mu_{\xi}(\tau_z\w,tQ)}{|tQ|}\leq \phi_\xi(\w)+(|\xi|+1)\mathbb{E}[\kappa(\cdot,0)]\frac{|Q''|-|Q'|}{|Q|}.	
	\end{equation*}
	Then, letting $Q'\nearrow Q$ and $Q''\searrow Q$ gives
	\begin{equation*}
		\limsup_{t\to +\infty}\frac{\mu_{\xi}(\tau_z\w,tQ)}{|tQ|}\leq \phi_{\xi}(\w).	
	\end{equation*}
 Analogously it can be proven that  
	\begin{equation*}
		\phi_{\xi}(\w) \leq \liminf_{t\to +\infty}\frac{\mu_{\xi}(\tau_z\w,tQ)}{|tQ|}.	
	\end{equation*}
which allows us to conclude both that $\tau_z\w \in \Omega'$ and that \eqref{c:invariance} holds true, so that $\phi_\xi$ is deterministic as claimed. Eventually, we define $f_{\rm hom}(\xi):=\phi_\xi$, for every $\xi\in \R^{m \times d}$.
	
We now show that $f_{\rm hom}$ satisfies the desired bounds. 	
	
Thanks to the Fatou Lemma we have that  
	\begin{equation*}
		f_{\rm hom}(\xi)=\mathbb{E}[f_{\rm hom}(\xi)]\leq \liminf_{t\to +\infty}\frac{1}{|tQ|}\mathbb{E}[\mu_{\xi}(\cdot,tQ)]\leq \sup_{|\eta|=1}\mathbb{E}[|\eta \Am(\cdot,0)|]\,|\xi|+\mathbb{E}[\lambda(\cdot,0)],
	\end{equation*}
	where to establish the last inequality we have used \eqref{eq:mu_bound}. Hence, the upper bound on $f_{\rm hom}$ is achieved. 
	
	To prove the lower bound, we observe that for any $v\in \ell_\xi +W_0^{1,1}(Q,\R^m)$ we have
		\begin{equation*}
		|tQ| |\xi|=\left|\int_{tQ} \nabla v\dx\right|\leq \int_{tQ} |\nabla v|\dx\leq \frac{\||\Am(\cdot,0)^{-1}|\|_{L^{\infty}(\Omega)}}{\alpha}\int_{tQ} f(\omega, x,\nabla v)\dx,
	\end{equation*}
where in the last estimate we used \eqref{eq:linear_from_below} with the actual constant $C=\||\Am(\cdot,0)^{-1}|\|_{L^{\infty}(\Omega)}$. Therefore by the arbitrariness of $v$, dividing by $|tQ|$ and passing to the limit as $t \to +\infty$ we get
	\begin{equation*}
		\alpha \||\Am(\cdot,0)^{-1}|\|_{L^{\infty}(\Omega)}^{-1}|\xi|\leq f_{\rm hom}(\xi),
	\end{equation*}
	thus the desired lower bound.
	
	Eventually, \eqref{c:det} follows by \cite[Theorem 2.3]{Kr}, hence the proof is accomplished.   
\end{proof}

\section{$\Gamma$-convergence}\label{s.proofs}
In this section we prove the $\Gamma$-convergence statement in Theorem \ref{thm.Gamma_pure}.
To do so, we start by establishing a compactness result for sequences $(u_\e)$ with equi-bounded energy $F_\e$. 

\begin{lemma}[Domain of the $\Gamma$-limit]\label{l.compactness}
Let $A\in \A$ and let $(u_{\e})\subset L^1_{\rm loc}(\R^d,\R^m)$ be such that
\begin{equation*}
\sup_{\e>0}\big(\|u_\e\|_{L^1(A,\R^m)}+ F_{\e}(\w)(u_{\e},A)\big)<+\infty.
\end{equation*}
Then there exists $u\in BV(A,\R^m)$ such that, up to subsequences, $u_{\e}\overset{*}{\rightharpoonup} u$ in $BV(A,\R^m)$. 
\end{lemma}
\begin{proof}
In view of \eqref{eq:linear_from_below} the sequence $(u_\e)$ is equi-bounded in the $W^{1,1}(A,\R^m)$-norm. Therefore the claim
follows by well-known compactness properties of $BV$-functions. 
\end{proof}

Below we prove two technical results which will be needed in what follows. The first one is a classical vectorial truncation result. 

\begin{lemma}[Vectorial truncation]\label{l.truncation}
	Let $A\in \A$ and let $u_{\e}\in L^1_{\rm loc}(\R^d,\R^m)\cap W^{1,1}(A,\R^m)$ be such that $u_{\e}\to u$ in $L^1_{\rm loc}(\R^d,\R^m)$, as $\e\to 0^+$. Then for every $\eta>0$ there exists a constant $C_{\eta}>0$ and a function $u_{\e,\eta}\in L^1_{\rm loc}(\R^d,\R^m)\cap W^{1,1}(A,\R^m)$ such that $u_{\e,\eta}=u_{\e}$ a.e. in $\{|u_{\e}|\leq \eta^{-1}\}$, and 
	\begin{equation*}
		|u_{\e,\eta}|\leq |u_{\e}|,\quad\quad |u_{\e,\eta}|\leq C_{\eta} \quad \text{for a.e. $x\in \R^d$.}
	\end{equation*}
	Moreover, there holds 
	\begin{equation*}
		F_{\e}(\w)(u_{\e,\eta},A)\leq (1+\eta)F_{\e}(\w)(u_{\e},A)+\eta,
	\end{equation*}
for every $\e>0$ small enough. 	
\end{lemma}
\begin{proof}
The proof is identical to that of \cite[Lemma 4.6]{RR21} up to replacing the exponent $p$ with $1$ and noting that the energy bound in  \cite[Lemma 4.6]{RR21} can be obtained simultaneously for all open sets $A\in \A$.
\end{proof}
Next we show that the functionals $F_\e$ satisfy a so-called fundamental estimate, uniformly in $\e$.
\begin{lemma}[Fundamental estimate]\label{l.fundamental_estimate}
Let $F_\e$ be as in \eqref{F-e}. Let $A, A', A'', B \in \A$ with $A'\subset \subset A''\subset A$ and let $u,v \in W^{1,1}(A,\R^m)$. For every $\delta>0$ there exists $\varphi \in C^\infty_c(A'',[0,1])$ (depending on $\delta$, $A'$ and $A''$) with $\varphi=1$ in a neighbourhood of $\overline{A'}$ such that for every $\e>0$ there holds
\begin{align}\nonumber
F_{\e}(\w)(\varphi u + (1-\varphi)v,A'\cup B)
&\leq (1+\delta)\Big(F_{\e}(\w)(u,A'')+F_{\e}(\w)(v,B)\Big)
\\ \label{c:fund-est}
&+\frac{4}{{\rm dist}(A', \partial A'')} \int_{S}|u-v||\Am(\w,\tfrac{x}{\e})|dx+\delta \int_S\lambda(\w,\tfrac{x}{\e})\dx,
\end{align}
where $S:=(A''\setminus \overline{A'})\cap B$. 
\end{lemma}
\begin{proof}
Let $\delta>0$ and let $A,A',A'',B$ and $S$ be as in the statement. Let $N=N(\delta) \in \mathbb N$ be such that
\begin{equation}\label{c:N}
\frac{1}{N} \max\Big\{\frac{1}{\a},1\Big\} \leq \delta. 
\end{equation}
Set $R:=\frac{1}{2}{\rm dist}(A',\partial A'')>0$. For $i=0,\ldots,N$ define
\[
A_i:=\Big\{x\in A'' \colon {\rm dist}(x,A')< \frac{i}{N} R\Big\}, 
\]
we have
\[
A'=:A_0\subset \subset A_1\subset \subset \ldots  \subset \subset A_{N} \subset \subset A''. 
\]
and for $i=0,\ldots,N-1$ let $\varphi_i\in C^\infty_c(A,[0,1])$ be such that ${\rm supp}\, \varphi_i \subset A_{i+1}$ and $\varphi_i=1$ in a neighbourhood of $\overline{A_i}$. We notice that $\varphi$ can be chosen in a way such that  $\|\nabla\varphi_i\|_{\infty} \leq 2N/R$.  

By virtue of the nonnegativity of $f$, for every $\e>0$ and for $i=0,\dots,N-1$ we have 
\begin{align}\nonumber
& F_{\e}(\w)(\varphi_i u +(1-\varphi_i)v,A'\cup B)
\\\nonumber
& =F^*_{\e}(\w)(u,(A''\cup B)\cap \overline{A_i})+F^*_{\e}(\w)(v, B\setminus A_{i+1})+F_{\e}(\w)(\varphi_i u +(1-\varphi_i)v,(A_{i+1}\setminus \overline{A_i}) \cap B)\nonumber
\\
&\leq F_{\e}(\w)(u,A'')+ +F_{\e}(\w)(v,B)+F_{\e}(\w)(\varphi_i u +(1-\varphi_i)v,S_i),\label{eq:split}
\end{align}
where $F^*_\e$ denotes the extension of $F_\e$ to the Borel subsets of $\R^d$ and $S_i:=(A_{i+1}\setminus \overline{A_i})\cap B$.

We now estimate the last term in \eqref{eq:split}. Since
\[
\nabla \big(\varphi_i u +(1-\varphi_i)v\big) =\nabla\varphi_i\otimes (u-v)+\varphi_i\nabla u+(1-\varphi_i)\nabla v,  
\]
by the upper bound in \eqref{gc} we get 
\begin{align*}
&\quad F_{\e}(\w)(\varphi_i u +(1-\varphi_i)v,S_i)\leq\int_{S_i}\Big(|\nabla \big(\varphi_i u +(1-\varphi_i)v\big) \Am(\w,\tfrac{x}{\e})|+\lambda(\w,\tfrac{x}{\e})\Big)\dx
\\
&\leq \int_{S_i}\Big(\frac{2N}{R} |u-v||\Lambda(\w,\tfrac{x}{\e})|+|\nabla u \Am(\w,\tfrac{x}{\e})|+|\nabla v \Am(\w,\tfrac{x}{\e})|+\lambda(\w,\tfrac{x}{\e})\Big)\dx.
\end{align*}
Now appealing to the lower bound in \eqref{gc} we obtain  
\begin{equation}\label{eq:intermediate_step}
F_{\e}(\w)(\varphi_i u +(1-\varphi_i)v,S_i)\leq \frac{1}{\alpha} \Big(F_{\e}(\w)(u,S_i)+F_{\e}(\w)(v,S_i)\Big)
 + \int_{S_i}\Big(\frac{2N}{R} |u-v||\Am(\w,\tfrac{x}{\e})|+\lambda(\w,\tfrac{x}{\e})\Big)\dx,
\end{equation}
for every $\e>0$ and for every $i=0,\ldots,N-1$. Hence by \eqref{eq:intermediate_step} there exists $i^\star \in \{0,\ldots,N-1\}$ such that 
\begin{align*}
&F_{\e}(\w)(\varphi_{i^\star} u +(1-\varphi_{i^\star})v,S_{i^\star})\leq \frac{1}{N}\sum_{i=0}^{N-1} F_{\e}(\w)(\varphi_{i} u +(1-\varphi_{i})v,S_{i})
\\
&\leq \frac{1}{N\a}\Big(F_{\e}(\w)(u,S)+F_{\e}(\w)(v,S)\Big) +\frac{2}{R} \int_S|u-v||\Am(\w,\tfrac{x}{\e})|\dx+\frac{1}{N}\int_S \lambda(\w,\tfrac{x}{\e})\dx.
\end{align*} 
Finally, thanks to \eqref{c:N}, the definition of $R$, and \eqref{eq:split} the desired estimate follows by choosing $\varphi=\varphi_{i^\star}$. 
\end{proof}
\begin{remark}\label{rmk:sf}
Let $A,A',A'', B \in \mathcal A$ be as in the statement of Lemma \ref{l.fundamental_estimate}. Let $(u_{\e}),(v_\e) \subset W^{1,1}(A,\R^m)$ be such that $(u_{\e},v_{\e})\to (u,v)$ in $L^1(A,\R^m) \times L^1(A,\R^m)$. Assume moreover that $(u_{\e}), (v_{\e})$ are uniformly bounded in the $L^{\infty}(A,\R^m)$-norm.  Then, there exists a set $\Omega'' \in \F$ with $\mathbb P(\Omega'')=1$ such that for every $\omega \in \Omega''$ and for every $\delta>0$ there exists a sequence $(w_{\e,\delta}) \subset W^{1,1}(A,\R^m)$ with $w_{\e,\delta}=u_\e$ in $A'$ and $w_{\e,\delta}=v_\e$ on $\partial A''$ such that 
\begin{align}\nonumber
\liminf_{\e\to 0^+}F_{\e}(\w)(w_{\e,\delta},A'\cup B)&\leq (1+\delta)\liminf_{\e\to 0^+} \Big(F_{\e}(\w)(u_{\e},A'')+F_{\e}(\w)(v_{\e},B)\Big)
\\\label{c:rmk-sf}
&\quad+\frac{4 \mathbb{E}[|\Am(\cdot,0)|]}{\dist(A',\partial A'')}\int_{(A''\setminus \overline{A'})\cap B}|u-v|\dx + \delta \mathbb{E}[\lambda(\cdot,0)]| |(A''\setminus {A'})\cap B|,
\end{align}
Moreover, the same estimate holds true if we replace $\liminf$ by $\limsup$.

In fact, up to a subsequence,  $|u_{\e}-v_{\e}|$ converges a.e. to $|u-v|$ and is uniformly bounded in the $L^{\infty}(A,\R^m)$-norm. Moreover, due to Theorem \ref{thm.additiv_ergodic} we know that, almost surely, $|\Am(\w,\tfrac{x}{\e})|$ and $\lambda(\w,\tfrac{x}{\e})$ converge weakly in $L^1(A)$ to $\mathbb{E}[|\Am(\cdot,0)|]$ and $\mathbb{E}[\lambda(\cdot,0)]$, respectively. By \cite[Proposition 2.61]{FoLe} we then have that, almost surely,   
\[
|u_\e-v_\e||\Am(\w,\tfrac{\cdot}{\e})|\, \rightharpoonup \, \mathbb{E}[|\Am(\cdot,0)|]|u-v|,
\]
weakly in $L^1(A)$. Therefore \eqref{c:rmk-sf} follows by Lemma \ref{l.fundamental_estimate} setting $w_{\e,\delta}:=\varphi u_\e + (1-\varphi)v_\e$. 
\end{remark}
We are now in a position to establish the $\Gamma$-convergence result for the functionals $F_\e$. This is done by proving the liminf and limsup inequalities separately. We start with the limsup-inequality whose proof relies on a density and relaxation argument.

\begin{proposition}\label{p.ub}
Let $F_\e$ and $F_{\rm hom}$ be as in \eqref{F-e} and \eqref{F-hom}, respectively. Then, there exists $\widetilde \Omega \in \F$ with $\mathbb P(\widetilde \Omega)=1$ such that for every $\omega \in \widetilde \Omega$, every $u \in L^1_{\rm loc}(\R^d,\R^m)$, there exists a sequence $(u_\e) \subset L^1_{\rm loc}(\R^d,\R^m)$ with $u_\e \to u$ in $L^1_{\rm loc}(\R^d,\R^m)$ satisfying   
\begin{equation}\label{c:ls}
\limsup_{\e\to 0^+}F_{\e}(\w)(u_{\e},A) \leq F_{\rm hom}(u,A),	
\end{equation}
for every $A\in \mathcal A$. 
\end{proposition}
\begin{proof}
Let $u \in L^1_{\rm loc}(\R^d,\R^m)$ and $A\in \A$, moreover assume that $u\in BV(A,\R^m)$, otherwise there is nothing to prove. 

We recall that \eqref{c:ls} is equivalent to proving that for every  $u\in BV(A,\R^m)$ and every $A\in \A$ there holds
\begin{equation}\label{G-limsup}
F''(\omega)(u,A)\leq F_{\rm hom}(u,A),
\end{equation}
where $F''(\w):L^1_{\rm loc}(\R^n,\R^m)\times \A\longrightarrow [0,+\infty]$ is defined as
\begin{equation}\label{F''}
F''(\w)(u,A):=\inf\{\limsup_{\e\to 0^+}F_{\e}(\w)(u_{\e},A):\,u_{\e}\to u\text{ in }L^1_{\rm loc}(\R^d,\R^m)\}.
\end{equation}
Moreover, it is well-known that $F''(\w)(\cdot,A)$ is $L^1_{\rm loc}(\R^d,\R^m)$-lower semicontinuous. 

Let $\Omega', \Omega'' \in \F$ be the sets of full probability whose existence is established by Lemma \ref{l.existence_f_hom} and Remark \ref{rmk:sf}, respectively. Set $\widetilde \Omega:=\Omega' \cap \Omega''$; clearly $\mathbb P(\widetilde \Omega)=1$. Throughout the proof $\omega$ is arbitrarily fixed in $\widetilde\Omega$.

We achieve the proof of \eqref{G-limsup} in three steps. 

\medskip

\textbf{Step 1:} \emph{Proof of \eqref{G-limsup} for $u\in W^{1,1}(A,\R^m)$.}

\smallskip

Let $A\in \A$ be fixed; in this step we show that
\begin{equation}\label{eq:Gamma-limsup}
F''(\w)(u,A)\leq \int_A f_{\rm hom}(\nabla u)\dx,	
\end{equation}
for all $u\in W^{1,1}(A,\R^m)$. Since $A$ has Lipschitz boundary, it is not restrictive to assume that $u\in W^{1,1}(\R^d,\R^m)$.

We observe that by the continuity and the linear growth of $f_{\rm hom}$ (cf. Lemma \ref{l.existence_f_hom}) the functional in the right-hand side of \eqref{eq:Gamma-limsup} is continuous with respect to strong $W^{1,1}(A,\R^m)$-convergence. Then, by standard density arguments it suffices to prove \eqref{eq:Gamma-limsup} for (continuous) piecewise affine functions.
That is, we can assume that $u$ is continuous and that there exists a locally finite triangulation $\{T_i\}_{i\in\N}$ of $\R^d$ in non-degenerate $(d+1)$-simplices such that $u|_{T_i}$ is affine for every $i\in\N$. 

To construct a recovery sequence $(u_\e)$ for such a $u$ we first construct it locally, in each $T_i$, in a way so that $u_\e^i:=u_{\e}|_{T_i}\in u+W^{1,1}_0(T_i,\R^m)$, for all $i\in\N$. Then, thanks to the continuity of $u$, the locally defined sequences $(u_\e^i)$ can be ``glued'' together to obtain a recovery sequence $(u_{\e})$ defined on the whole $\R^d$. 

To this end, we first focus on a single simplex $T_i$ and write $u|_{T_i}=\ell_{\xi_i}+b_i$, for some $\xi_i\in \R^{m \times d}$ and $b_i\in \R^m$. For $\delta>0$ small, we consider the family of pairwise disjoint open cubes of side-length $\delta$ contained in $T_i$ defined as 
\begin{equation*}
\mathcal{Q}_i^{\delta}:=\{Q^{\delta}:=Q_\delta(\delta z) \colon z\in\Z^d,\,Q^\delta\subset T_i\}
\end{equation*}
and the inner approximation of $T_i$ defined as $T_i^\delta:=\bigcup_{Q^\delta \in \mathcal{Q}_i^{\delta}} Q^\delta$. We define the sequence $(u^i_{\e})$ separately in each cube in $\mathcal{Q}_i^{\delta}$. Then, for $\e>0$ and $Q^\delta\in\mathcal{Q}_i^{\delta}$ fixed let $v^i_{\e,Q^\delta}\in \ell_{\xi_i}+W^{1,1}_0(\e^{-1}Q^\delta,\R^m)$ satisfy
\begin{equation*}
\int_{\e^{-1}Q^\delta}f(\w,x,\nabla v^i_{\e,Q^\delta})\dx\leq \mu_{\xi_i}(\w,\e^{-1}Q^\delta)+\e,
\end{equation*}
where $\mu_{\xi_i}$ is as in \eqref{c:mu-xi} with $\xi$ replaced by $\xi_i$. Set $u^i_{\e,Q^\delta}:=v^i_{\e,Q^\delta}-\ell_{\xi_i}$, thus $u^i_{\e,Q^\delta}\in W^{1,1}_0(\e^{-1}Q^\delta,\R^m)$.
We then define $u^i_{\e}$ on $T_i$ as 
\[
u^i_{\e}(x):=u|_{T_i}+\sum_{Q^\delta\in\mathcal{Q}_i^{\delta}}\e u^i_{\e, Q^\delta}(x/\e)\chi_{Q^\delta}(x);
\]
we notice that $u^i_\e$ depends also on $\delta$. 

Thanks to the boundary conditions satisfied by $u^i_{\e, Q^\delta}$ we clearly have that $u^i_{\e}\in u+W^{1,1}_0(T_i,\R^m)$. 

By the upper bound in \eqref{gc}, we can estimate the energy of $u^i_{\e}$ on the simplex $T_i$ as follows
\begin{align*}
F_{\e}(\w)(u^i_{\e},T_i)&= \sum_{Q^\delta\in\mathcal{Q}^{\delta}_i}\int_{Q^\delta} f(\w,\tfrac{x}{\e},\nabla u^i_{\e,Q^\delta}(\tfrac{x}{\e})+\xi_i)\dx+\int_{T_i\setminus T^{\delta}_i}f(\w,\tfrac{x}{\e},\xi_i)\dx
\\
&\leq \sum_{Q^\delta\in\mathcal{Q}_i^{\delta}}\e^{d}\int_{\e^{-1}Q^\delta}f(\w,y,\nabla v^i_{\e,Q^\delta})\dy+\e^d\int_{\e^{-1}(T_i\setminus T_i^{\delta})}\big(|\xi_i||\Lambda(\w,y)|+\lambda(\w,y)\big)\dy
\\
&\leq \sum_{Q^\delta\in\mathcal{Q}_i^{\delta}} |Q^\delta|\,\frac{\mu_{\xi_i}(\omega,\e^{-1}Q^\delta)}{\e^{-d} |Q^\delta|}+\e^d\int_{\e^{-1}(T_i\setminus T_i^{\delta})}\big(|\xi_i||\Lambda(\w,y)|+\lambda(\w,y)\big)\dy +o(1),
\end{align*}
as $\e \to 0^+$, where to establish the last inequality we have used the definition of $v^i_{\e,Q^\delta}$. 

Since $\omega \in \widetilde\Omega$, Lemma \ref{l.existence_f_hom} ensures that
\begin{equation*}
\lim_{\e\to 0^+}\frac{\mu_{\xi_i}(\w,\e^{-1}Q^\delta)}{\e^{-d}|Q^{\delta}|}= f_{\rm hom}(\xi_i),	
\end{equation*}
while Theorem \ref{thm.additiv_ergodic} applied to $|\Lambda|+\lambda$ yields
\begin{equation*}
\lim_{\e\to 0^+}\e^d\int_{\e^{-1}(T_i\setminus T^{\delta}_i)}|\Lambda(\w,y)|+\lambda(\w,y)\dy=|T_i\setminus T_i^{\delta}|\,\mathbb{E}[|\Lambda(\cdot,0)|+\lambda(\cdot,0)].
\end{equation*}
Therefore we get 
\begin{align}\nonumber 
	\limsup_{\e\to 0^+}F_{\e}(\w)(u^i_{\e},T_i)&\leq  \sum_{Q^\delta\in\mathcal{Q}_i^{\delta}}|Q^\delta|f_{\rm hom}(\xi_i)+(|\xi_i|+1)|T_i\setminus T_i^{\delta}|\,\mathbb{E}[|\Lambda(\cdot,0)|+\lambda(\cdot,0)]	
	\\\label{eq:almost_recovery}
	&\leq \int_{T_i} f_{\rm hom}(\nabla u)\dy+o(1),
\end{align}
as $\delta \to 0^+$. 
Set 
\[
u_\e:=\begin{cases}
u_\e^i & \text{in $T_i$,\, if $T_i\cap A \neq \emptyset$} 
\cr
u & \text{otherwise in $\R^d$};
\end{cases}
\]
appealing to \eqref{eq:almost_recovery} we have that
\begin{align}\nonumber
\limsup_{\e\to 0^+}F_{\e}(\w)(u_{\e},A) &\leq \hspace{-0.2cm} \sum_{i \colon T_i\cap A\neq\emptyset}\limsup_{\e\to 0^+}F_{\e}(\w)(u^i_{\e},T_i)\leq \hspace{-0.3cm} \sum_{i \colon T_i\cap A\neq\emptyset}\int_{T_i}f_{\rm hom}(\nabla u)\dx + o(1)
\\\label{c:almost-ls}
&\leq \int_A f_{\rm hom}(\nabla u)\dx \, + \hspace{-0.3cm} \sum_{i \colon T_i\cap \partial A\neq\emptyset}\int_{T_i}f_{\rm hom}(\nabla u)\dx + o(1),
\end{align}
as $\delta \to 0^+$.

We now analyse the asymptotic behaviour of $u_{\e}$; to this end we recall that $u_\e$ also depends on $\delta$. In view of  \eqref{c:almost-ls} we can combine Poincar\'e's Inequality and Lemma \ref{l.compactness} to infer that $u_{\e}$ is bounded in $W^{1,1}(A,\R^m)$, uniformly in $\e$. Hence, also taking into account the definition of $u_\e$ up to a subsequences (not relabelled) $u_{\e}\to {u}_{\delta}$ in $L^1_{\rm loc}(\R^d,\R^m)$, with $u_{\delta}\in W^{1,1}(A,\R^m)$. We now estimate the difference between ${u}_{\delta}$ and the target function $u$ in $L^1(T_i,\R^m)$, for every $i$ such that $T_i\cap A \neq \emptyset$. By the Poincar\'e Inequality on the cubes $Q^\delta\in\mathcal{Q}_i^{\delta}$, we have 
\begin{align*}
\|u_{\delta}-u\|_{L^1(T_i)}&=\lim_{\e\to 0^+}\sum_{Q^\delta\in\mathcal{Q}_i^{\delta}} \int_{Q^\delta} |\e u^i_{\e,Q^\delta}(\tfrac{x}{\e})|\dx\leq C\delta\liminf_{\e\to 0^+}\sum_{Q^\delta\in\mathcal{Q}_i^{\delta}}\int_{Q^\delta}|\nabla u^i_{\e,Q^\delta}(\tfrac{x}{\e})|\dx
\\
&\leq C\delta \left(|\xi_i||T_i|+\liminf_{\e\to 0^+}\sum_{Q^\delta\in\mathcal{Q}_i^{\delta}}\int_{Q^\delta}|\xi_i+\nabla u^i_{\e,Q^\delta}(\tfrac{x}{\e})|\dx\right)
\\
&\leq  C\delta \left(|\xi_i||T_i|+\liminf_{\e\to 0^+}\int_{T_i}|\nabla u^i_{\e}|\dx\right).
\end{align*}
Using \eqref{eq:linear_from_below} we then have
\begin{equation*}
\|u_{\delta}-u\|_{L^1(T_i)}\leq C\delta\left(|\xi_i||T_i|+\limsup_{\e\to 0^+}F_{\e}(\w)(u^i_{\e},T_i)\right).
\end{equation*}
Thanks to \eqref{eq:almost_recovery} the term inside the parenthesis above is finite, and therefore we conclude that $u_{\delta}\to u$ in $L^1(T_i,\R^m)$ as $\delta\to 0$ and analogously, also taking into account the definition of $u_\e$ we have that $u_{\delta}\to u$ in $L^1_{\rm loc}(\R^d,\R^m)$.
 
Eventually, by the $L^1_{\rm loc}(\R^d,\R^m)$-lower semicontinuity of $F''(\w)(\cdot,A)$ and by \eqref{c:almost-ls} we obtain
\begin{align*}
F''(\w)(u), A&\leq\liminf_{\delta\to 0^+}F''(\w)(u_{\delta},A)\leq\limsup_{\delta\to 0^+}\limsup_{\e\to 0^+} F_{\e}(\w)(u_\e,A)
\\
&\leq \int_A f_{\rm hom}(\nabla u)\dx \, + \hspace{-0.3cm} \sum_{i \colon T_i\cap \partial A\neq\emptyset}\int_{T_i}f_{\rm hom}(\nabla u)\dx,
\end{align*}
hence the claim follows by a standard diagonal argument also refining the triangulation $\{T_i\}_{i\in \N}$ by choosing simplices of arbitrarily small diameter.

\medskip

\textbf{Step 2:} \emph{Quasiconvexity of $f_{\rm hom}$.}

\smallskip

In this step we prove that the function $f_{\rm hom}$ is quasiconvex, that is, we show that
\begin{equation}\label{eq:f_hom_quasiconvex}
	|Q|f_{\rm hom}(\xi)\leq \int_Q f_{\rm hom}(\xi+\nabla u)\dx, 
\end{equation}
for every $\xi\in\R^{m\times d}$ and every $u\in C^1_0(Q,\R^m)$, where $Q:=Q_\rho(x_0)\subset \R^d$ is an open cube.

Since $\ell_\xi+u\in L^{\infty}(Q,\R^m)$, Lemma \ref{l.truncation} together with Step 1 ensure that for $\eta>0$ small enough there exists $(u_{\e,\eta})\subset W^{1,1}(Q,\R^m)$, with $|u_{\e,\eta}|\leq C_{\eta}$ a.e. in $Q$, such that $u_{\e,\eta}\to \ell_\xi+u$ in $L^1(Q,\R^m)$,  and
\begin{equation*}
\limsup_{\e\to 0^+}F_{\e}(\w)(u_{\e,\eta},Q)\leq (1+\eta)\int_Q f_{\rm hom}(\xi+\nabla u)\dx+\eta.
\end{equation*}
Set $Q_{1-\eta}:=Q_{(1-\eta)\rho}(x_0)$; we notice that since $u\in C_0^{\infty}(Q,\R^m)$, then $u=0$ on $Q\setminus \overline{Q_{1-\eta}}$, for $\eta>0$ small enough.
We now invoke Remark \ref{rmk:sf} to modify $u_{\e,\eta}$ in a neighbourhood of $\partial Q$. Namely, choosing $u_{\e}=u_{\e,\eta}$, $v_{\e}=\ell_\xi$, $A'=Q_{1-\eta}$, $A''= Q$, and $B=Q\setminus \overline{Q_{1-\eta}}$ for every $\delta>0$ we get a sequence $w_{\e,\eta,\delta}\in \ell_\xi+W_0^{1,1}(Q,\R^m)$ such that $w_{\e,\eta,\delta}=u_{\e,\eta}$ on $Q_{1-\eta}$, whereas by \eqref{c:rmk-sf} there holds
\begin{equation}\label{c:c1}
	\limsup_{\e\to 0^+}F_{\e}(\w)(w_{\e,\eta,\delta},Q)\leq (1+\delta)\limsup_{\e\to 0^+}\Big(F_{\e}(\w)(u_{\e,\eta},Q)+F_{\e}(\w)(\ell_\xi,Q\setminus \overline{Q_{1-\eta}})\Big)+C\delta |Q\setminus \overline{Q_{1-\eta}}|.
\end{equation}
Since $\omega \in \widetilde \Omega$ by combining \eqref{gc} and Theorem \ref{thm.additiv_ergodic} we get
\begin{align}\nonumber
\limsup_{\e\to 0^+}F_{\e}(\w)(\ell_\xi,Q\setminus \overline{Q_{1-\eta}})&\leq \limsup_{\e\to 0^+}(|\xi|+1)\int_{Q\setminus \overline{Q_{1-\eta}}}|\Lambda(\w,\tfrac{x}{\e})|+\lambda(\w,\tfrac{x}{\e})\dx
\\\label{c:c2}
&=(|\xi|+1)\,\mathbb{E}[|\Lambda(\cdot,0)|+\lambda(\cdot,0)]\,|Q\setminus \overline{Q_{1-\eta}}|,
\end{align}
moreover, by a change of variables and Lemma \ref{l.existence_f_hom} we have
\begin{equation}\label{c:c3}
|Q|f_{\rm hom}(\xi)=\lim_{\e\to 0^+}\e^d\mu_{\xi}(\w,\e^{-1}Q) \leq \limsup_{\e\to 0^+}F_{\e}(\w)(w_{\e,\eta,\delta},Q).
\end{equation}
Finally, gathering \eqref{c:c1}-\eqref{c:c3} gives
\begin{align*}
|Q|f_{\rm hom}(\xi)\leq& (1+\delta)\left((1+\eta)\int_Q f_{\rm hom}(\xi+\nabla u)\dx+\eta+(|\xi|+1)\,\mathbb{E}[|\Lambda(\cdot,0)|+\lambda(\cdot,0)]\,|Q\setminus \overline{Q_{1-\eta}}|\right)
\\
&\quad+C\delta |Q\setminus \overline{Q_{1-\eta}}|.
\end{align*}  
Therefore \eqref{eq:f_hom_quasiconvex} follows by letting $\delta\to 0^+$ and $\eta\to 0^+$. 

\medskip

\textbf{Step 3:} \emph{Proof of \eqref{G-limsup} by relaxation.}

\smallskip

Since $f_{\rm hom}$ is non-negative, quasiconvex, and satisfies the upper bound $f_{\rm hom}(\xi)\leq C_0|\xi|+C_1$ we can invoke \cite[Theorem 4.1]{AmDM92} to deduce that for $A\in \A$ the $L^1_{\rm loc}(\R^d,\R^m)$-lower semicontinuous envelope of 
\begin{equation*}
	W^{1,1}(A,\R^m)\ni u\mapsto \int_A f_{\rm hom}(\nabla u)\dx
\end{equation*}
on $BV(A,\R^m)$ is given by
\begin{equation*}
F_{\rm hom}(u,A)=\int_A f_{\rm hom}(\nabla u)\dx+\int_A f_{\rm hom}^{\infty}\left(\frac{\mathrm{d}D^su}{\mathrm{d}|D^su|}\right)\,\mathrm{d}|D^su|.
\end{equation*}
Therefore \eqref{G-limsup} follows by taking the $L^1_{\rm loc}(\R^d,\R^m)$-lower-semicontinuous envelope of both sides in \eqref{eq:Gamma-limsup}.
\end{proof}

We now show that the liminf inequality holds true. In this case the proof is achieved by resorting to the Fonseca and M\"uller blow-up method \cite{FoMue} (see also \cite{AmDM92}).

\begin{proposition}\label{p.lb}
Let $F_\e$ and $F_{\rm hom}$ be as in \eqref{F-e} and \eqref{F-hom}, respectively. Then, there exists $\widetilde \Omega \in \F$ with $\mathbb P(\widetilde \Omega)=1$ such that for every $\omega \in \widetilde \Omega$, every $u \in L^1_{\rm loc}(\R^d,\R^m)$, and every sequence $(u_\e) \subset L^1_{\rm loc}(\R^d,\R^m)$ with $u_\e \to u$ in $L^1_{\rm loc}(\R^d,\R^m)$ there holds  
\begin{equation}\label{c:li}
F_{\rm hom}(u,A)\leq\liminf_{\e\to 0^+}F_{\e}(\w)(u_{\e},A),	
\end{equation}
for every $A\in \mathcal A$. 
\end{proposition}
\begin{proof}
Let $\Omega', \Omega'' \in \F$ be the sets of probability one whose existence is established by Lemma \ref{l.existence_f_hom} and Remark \ref{rmk:sf}, respectively. Set $\widetilde \Omega:=\Omega' \cap \Omega''$; clearly $\mathbb P(\widetilde \Omega)=1$. Throughout the proof $\omega$ is arbitrary, fixed, and belongs to $\widetilde\Omega$. 
 
Let $u \in L^1_{\rm loc}(\R^d,\R^m)$ and let $(u_\e) \subset L^1_{\rm loc}(\R^d,\R^m)$ be such that $u_\e \to u$ in $L^1_{\rm loc}(\R^d,\R^m)$. Let $A\in \mathcal A$ and suppose that 
\begin{equation}\label{c:li-bd}
\liminf_{\e\to 0^+}F_{\e}(\w)(u_{\e},A)<+\infty
\end{equation}
(otherwise there is nothing to prove). Moreover, up to subsequences, we can also assume that the liminf in \eqref{c:li-bd} is actually a limit.  
Thanks to \eqref{c:li-bd}, we immediately deduce that $(u_\e)\subset W^{1,1}(A,\R^m)$, moreover by Lemma \ref{l.compactness} we know that $u\in BV(A,\R^m)$.   

The proof of \eqref{c:li} is carried out in two main steps.

\medskip

{\bf Step 1:} \emph{Proof of \eqref{c:li} for sequences $(u_\e)$ equi-bounded in $L^{\infty}(A,\R^m)$.} 

\smallskip

Assume that there exists $M<+\infty$ such that for every $\e>0$
\begin{equation}\label{eq:uniformbound}
\|u_{\e}\|_{L^{\infty}(A,\R^m)}\leq M. 
\end{equation} 
For $\omega \in \Omega$ fixed and for every Borel subset $B$ of $A$ we define the finite Radon-measures $\nu_{\e}$ as
\begin{equation*}
\nu_{\e}(\omega, B):=\int_B f(\w,\tfrac{x}{\e},\nabla u_{\e})\dx.
\end{equation*}

By \eqref{c:li-bd} the total variation of the sequence $(\nu_{\e})$ is equi-bounded, therefore, up to subsequences, we have that  $\nu_{\e}\overset{\star}{\rightharpoonup}\nu$, for some nonnegative finite Radon measure $\nu$. By the Lebesgue Decomposition Theorem, we can write $\nu= \nu^a+ \nu^s$, where $\nu^a$ and $\nu^s$ are, respectively, absolutely continuous and singular with respect to the Lebesgue measure. We then have 
\[
\nu^a = \tilde f(\omega, x) \dx,
\] 
for some nonnegative integrable function $\tilde f$.   

Since $A$ is open, the weak$^\star$ convergence of $\nu_\e$ to $\nu$ implies that
\begin{equation*}
\liminf_{\e\to 0^+}F_{\e}(\w)(u_{\e},A)=\liminf_{\e\to 0^+}\nu_{\e}(\omega, A)\geq \nu(\omega, A)= \int_{A}\tilde{f}(\omega,x)\dx+\nu^s(\omega,A).
\end{equation*} 
We now separately estimate from below the integrand $\tilde{f}$ and the measure $\nu^s$. Since $\frac{\D u}{|\D u|}=\frac{\D^s u}{|\D^s u|}$ for $|\D^s u|$-a.e. $x_0\in A$, to prove \eqref{c:li} it suffices to show that for every $\omega \in \widetilde \Omega$
\begin{align}\label{c:volume}
	\widetilde{f}(\omega, x_0)&\geq f_{\rm hom}(\nabla u(x_0))\qquad\,\text{ for a.e. }x_0\in A,
	\\\label{c:cantor}
	\frac{{\rm d}\nu^s}{{\rm d}|\D u|}(\omega, x_0)&\geq f^\infty_{\rm hom}\Big(\frac{{\rm d}\D u}{{\rm d} |\D u|}(x_0)\Big)\;\text{ for }|\D^s u|\text{-a.e. }x_0\in A,
\end{align}
where $f_{\rm hom}$ is as in Lemma \ref{l.existence_f_hom} and $f^\infty_{\rm hom}$ denotes its recession function. 
\smallskip

{\bf Substep 1.1:} \emph{Proof of \eqref{c:volume}.}

\smallskip

Let $x_0\in A$ and let $r>0$ be so small that $Q_r(x_0)\subset A$. Since $\nu$ is a finite Radon measure, it follows that $\nu(\omega,\partial Q_r(x_0))=0$ except for a countable number of radii. Then, the Besicovitch Differentiation Theorem \cite[Theorem 1.153]{FoLe} and the Portmanteau Theorem imply that for a.e. $x_0\in A$ (along a suitable sequence $r\to 0^+$) we have
\begin{equation*}
\tilde{f}(\omega,x_0)=\lim_{r\to 0^+}\frac{\nu(\omega, Q_{r}(x_0))}{r^d}= \lim_{r\to 0^+}\lim_{\e\to 0^+}\frac{\nu_{\e}(\omega,Q_{r}(x_0))}{r^d}.
\end{equation*}
Therefore to prove \eqref{c:volume} it suffices to show that for a.e. $x_0\in A$ we have
\begin{equation}\label{eq:blowup}
\liminf_{r\to 0^+}\liminf_{\e\to 0^+}\dashint_{Q_r(x_0)}f_{\e}(\w,\tfrac{x}{\e},\nabla u_{\e})\dx\geq f_{\rm hom}(\nabla u(x_0)).
\end{equation}
Let $x_0$ be a Lebesgue point of $u$ and $\nabla u$ and set $L_{u,x_0}(x):=u(x_0)+\nabla u(x_0)(x-x_0)$.

Due to \eqref{eq:uniformbound} we can invoke Remark \ref{rmk:sf} to modify $u_{\e}$ close to $\partial Q_r(x_0)$. Namely, Remark \ref{rmk:sf} applied to the sequences $u_{\e}$ and $v_{\e}=L_{u,x_0}\in L^{\infty}(A,\R^m)$ and to the open sets $A'= Q_{s r}(x_0)$, with $s\in (0,1)$, $A''=Q_r(x_0)$, and $B=Q_r(x_0) \setminus \overline{Q_{s r}(x_0)}$, for every $\delta>0$ provides us with $w_{\e,\delta}\in W^{1,1}(A,\R^m)$ satisfying $w_{\e,\delta}=u_{\e}$ on $Q_{sr}(x_0)$ and $w_{\e,\delta}=L_{u,x_0}$ on $\partial Q_r(x_0)$. Moreover, \eqref{c:rmk-sf} now reads as 
\begin{align}\nonumber
\liminf_{\e\to 0^+}F_{\e}(\w)(w_{\e,\delta},Q_{r}(x_0))&\leq (1+\delta)\liminf_{\e\to 0^+}\Big(F_{\e}(\w)(u_{\e},Q_r(x_0))+F_{\e}(\w)(L_{u,x_0},Q_{r}(x_0)\setminus \overline{Q_{sr}(x_0)})\Big)
\\\label{c:f}
&\quad+\frac{C}{(1-s)r}\int_{Q_{r}(x_0)\setminus Q_{sr}(x_0)}|u(x)-L_{u,x_0}(x)|\dx+C\delta |Q_{r}(x_0)\setminus Q_{sr}(x_0)|,
\end{align}
where used that $\dist(Q_{sr}(x_0),\partial Q_r(x_0))=\tfrac{1}{2}(1-s)r$. 

Since $w_{\e,\delta}$ is admissible as test function in the minimisation problem defining the ergodic process $\mu_{\nabla u(x_0)}$ and $\omega \in \widetilde \Omega$, invoking Lemma \ref{l.existence_f_hom} gives  
\begin{equation}\label{c:s}
	\liminf_{\e\to 0^+}F_{\e}(\w)(w_{\e,\delta},Q_{r}(x_0))\geq \lim_{\e\to 0^+}\e^d\mu_{\nabla u(x_0)}(\w,\e^{-1} Q_r(x_0))=f_{\rm hom}(\nabla u(x_0))r^d.
\end{equation}
Moreover by \eqref{gc}, appealing to Theorem \ref{thm.additiv_ergodic} we get 
\begin{align}\nonumber
\limsup_{\e\to 0^+}F_{\e}(\w)(L_{u,x_0},Q_{r}(x_0)\setminus \overline{Q_{sr}(x_0)})&\leq \lim_{\e\to 0^+}(|\nabla u(x_0)+1)\int_{Q_r(x_0)\setminus {Q_{sr}(x_0)}}\big(|\Am(\w,\tfrac{x}{\e})|+\lambda(\w,\tfrac{x}{\e})\big)\dx
\\\label{c:t}
&=(|\nabla u(x_0)|+1)\mathbb{E}[|\Lambda(\cdot,0)|+\lambda(\cdot,0)]|Q_r(x_0)\setminus Q_{sr}(x_0)|.
\end{align}
Gathering \eqref{c:f}-\eqref{c:t} yields
\begin{align*}
\frac{r^d}{1+\delta}f_{\rm hom}(\nabla u(x_0))&\leq \liminf_{\e\to 0}F_{\e}(\w)(u_{\e},Q_r(x_0))+C\delta|Q_{r}(x_0)\setminus Q_{sr}(x_0)|
\\
&\quad+\frac{C}{(1-s)r}\int_{Q_r(x_0)}|u(x)-L_{u,x_0}(x)|\dx.
\end{align*}
We now divide the above inequality by $r^d$; then recalling that due to \cite[Theorem 1, p. 228]{EvGa}
\[
\lim_{r\to 0^+}\frac{1}{r^{d+1}}\int_{Q_r(x_0)}|u(x)-L_{u,x_0}(x)|\dx=0,
\]
for a.e. $x_0\in A$, we first let $\delta\to 0^+$, then $r\to 0^+$, and eventually $s\to 1^-$ thus getting for a.e. $x_0\in A$ 
\begin{equation*}
f_{\rm hom}(\nabla u(x_0))\leq \liminf_{r\to 0^+}\liminf_{\e\to 0^+}\frac{1}{r^d}F_{\e}(\w)(u_{\e},Q_r(x_0)),
\end{equation*}
and thus the claim.

\smallskip

\noindent \textbf{Substep 1.2:} \emph{Proof of \eqref{c:cantor}.}

\smallskip

Thanks to \cite[Theorem 2.3]{AmDM92}, for $|\D^s u|$-a.e. $x_0\in A$ the following properties hold true:
	
	\smallskip
	
	$(i)$ $\frac{{\rm d}\D u}{{\rm d} |\D u|}(x_0)=\eta(x_0)\otimes n(x_0)$ for some $\eta(x_0)\in\R^m$ and $n(x_0)\in\R^d$ with $|\eta(x_0)|=|n(x_0)|=1$;
	
	\smallskip
	
	$(ii)$ setting $C_r(x_0)=x_0+r C$, for $r>0$ and $C$ a bounded, convex, open set with $0\in C$, there holds
	\begin{equation}\label{eq:measure_blow-up}
		\lim_{r\to 0^+}\frac{\D u(C_r(x_0))}{|\D u|(C_r(x_0))}=\eta(x_0)\otimes n(x_0),\quad\quad\lim_{r\to 0^+}\frac{|\D u|(C_r(x_0))}{r^d}=+\infty;
	\end{equation}
	
	\smallskip
	
	$(iii)$ defining $w_{r}\in BV(C,\R^m)$ as
	\begin{equation}\label{c:wr}
	w_{r}(y):=\frac{r^d}{|\D u|(C_r(x_0))} r^{-1}\left(u(x_0+r y)-\dashint_C u(x_0+r x)\dx\right),
	\end{equation}
	there exist a subsequence $r_h\to 0^+$, as $h \to+\infty$, and a function $w\in BV(C,\R^m)$ such that $w_{r_h}\to w$ in $L^1(C,\R^m)$. Moreover, $w$ can be represented as 
	\[
	w(y)=\psi(\langle y,n(x_0)\rangle)\eta(x_0),
	\] 
	where $\psi:(a,b)\to\R$ is a non-decreasing function with
	\begin{equation*}
	a:=\inf\{\langle y,n(x_0)\rangle:\,y\in C\},\quad\quad b:=\sup\{\langle y,n(x_0)\rangle:\,y\in C\}.
	\end{equation*}
	We notice that in the proof of $(iii)$ the passage to subsequences is only needed when using the compactness properties of bounded sequences in $BV$ and of bounded measures. Therefore, we can apply $(iii)$ along a further subsequence of radii which we choose below, to invoke Portmanteau's Theorem.

Let $x_0 \in A$ be fixed and such that properties $(i)\hbox{-}(iii)$ hold true. 
 
Set $n:=n(x_0)\in \mathbb S^{d-1}$; we complete the vector $n$ to an orthonormal basis $n_1,\ldots,n_{d-1},n$ of $\R^d$. 
In the same spirit as in the proof of \cite[Lemma 3.9]{BFM}, we choose the convex set $C$ to be: 
\begin{equation*}
C^k:=\{x\in\R^d:\,|\langle x,n\rangle|<1/2,\,|\langle x,n_i\rangle|<k/2 \quad\text{ for all }1\leq i\leq d-1\},
\end{equation*}
for $k\in\N$. With this choice there holds that $a=-\tfrac{1}{2}$ and $b=\tfrac{1}{2}$. Moreover set $C_r^k(x_0):=rC^k(x_0)=x_0+r C^k$ and let $w_r^k$ be as in \eqref{c:wr} with $C$ replaced by $C^k$. 

Again invoking the Besicovitch Differentiation Theorem \cite[Theorem 1.153]{FoLe}, also using \eqref{eq:measure_blow-up},  and the Portmanteau Theorem we can assume that along a sequence $r_h\to 0^+$
\begin{equation*}
	\frac{{\rm d}\nu^s}{{\rm d}|\D u|}(\omega, x_0)=\lim_{h\to +\infty}\frac{\nu(\omega, C^k_{r_h}(x_0))}{|\D u|(C^k_{r_h}(x_0))}=\lim_{h\to +\infty}\lim_{\e\to 0^+}\frac{\nu_{\e}(\omega, C^k_{r_h}(x_0))}{|\D u|(C^k_{r_h}(x_0))}.
\end{equation*}
Hence, in view of $(i)$ to get \eqref{c:cantor} it suffices to show that
\begin{equation}\label{eq:claim_singular}
	\limsup_{k\to +\infty}\limsup_{h\to +\infty}\liminf_{\e\to 0^+}\frac{1}{|\D u|(C^k_{r_h}(x_0))}	\int_{C^k_{r_h}(x_0)}f(\w,\tfrac{x}{\e},\nabla u_{\e}(x))\dx\geq f_{\rm hom}^{\infty}(\eta\otimes n).
\end{equation}

Up to refining the subsequence $r_h\to 0^+$, by (iii) we find a function $w^k\in BV(C^k,\R^m)$ such that $w^k_{r_h}\to w^k$ in $L^1(C^,\R^m)$; furthermore  $w^k$ can be represented as 
	\[
	w^k(y)=\psi^k(\langle y,n(x_0)\rangle)\eta(x_0),
	\] 
for some non-decreasing function $\psi^k \colon (-\frac{1}{2}, \frac{1}{2}) \to \R$.  Since $|\eta(x_0)\otimes n(x_0)|=1$, by slicing we have
\begin{equation}\label{c:mass-Dwk}
|Dw^k|(C^k(x_0))=k^{d-1}\left(\psi^k(\tfrac{1}{2}^-)-\psi^k(-\tfrac{1}{2}^+)\right).
\end{equation} 
As shown in the proof of \cite[Theorem 2.3]{AmDM92}, the measure $|\D w^k|$ coincides with the weak$^*$-limit of the measures $|\D w^k_{r_h}|$. Thanks to \cite[Lemma 5.1]{L} we get $|\D w^k|(C^k(x_0))=1$, which by \eqref{c:mass-Dwk} yields 
\begin{equation}\label{eq:jump_amplitude}
\left(\psi^k(\tfrac{1}{2}^-)-\psi^k(-\tfrac{1}{2}^+)\right)=k^{1-d}.
\end{equation}
Moreover, since $w^k_r$ has zero mean-value we get that
\begin{equation}\label{eq:meanzero}
	\int_{-1/2}^{1/2}\psi^k(t)\,\mathrm{d}t=0.
\end{equation}

Set $t_h^k:=\tfrac{|\D u|(C_{r_h}^k(x_0))}{r_h^d}$. By \eqref{eq:measure_blow-up} we have that 
\begin{equation}\label{eq:scale_t}
	\lim_{r\to +\infty}t_h^k=+\infty.
\end{equation}
Moreover, thanks to \eqref{eq:uniformbound}, we can apply \eqref{c:rmk-sf} choosing the functions $u_{\e}$,
\begin{equation*}
v_{k}(x):=t_h^k\left(k^{1-d}\eta\langle n,x-x_0\rangle+ r_h\eta\frac{\psi^k(\tfrac{1}{2}^-)+\psi^k(-\tfrac{1}{2}^+)}{2}\right)+\dashint_{C_{r_h}^k(x_0)}u(y)\dy,
\end{equation*}
where $\eta:=\eta(x_0)$, and the sets 
\begin{equation*}
A''=C_{r_h}^k(x_0),\quad 	B=C^k_{r_h,s}(x_0),\quad A'=C_{r_h}^k(x_0)\setminus \overline{C_{r_h,s}^k(x_0)},
\end{equation*}
where for $s\in (0,1)$ we define the anisotropic annular set
\begin{equation*}
C^k_{r,s}(x_0):=x_0+\left(\{x\in C_r^k:\,\max_{1\leq i\leq d-1}\{|\langle x,n_i\rangle|\}>(k-s)r/2\}\cup\{x\in C_r^k:\,\,|\langle x,n\rangle|>(1-s)r/2\}\right).
\end{equation*}
Since $\omega\in \widetilde \Omega$ and $\nabla v_k=t_h^k k^{1-d}\eta\otimes n$, from Lemma \ref{l.existence_f_hom} we deduce that for any $\delta>0$
\begin{align}\label{eq:applying_fund_est}
|C_{r_h}^k(x_0)|f_{\rm hom}(t_h^{k} k^{1-d}\eta\otimes n)&\leq (1+\delta)\liminf_{\e\to 0^+} \Big(F_{\e}(\w)(u_{\e},C_{r_h}^k(x_0))+F_{\e}(\w)(v_k,C^k_{r_h,s}(x_0))\Big)\nonumber
\\
&\quad+\frac{C}{sr}\int_{C^k_{r_h,s}(x_0)}|u-v_k|\dx+C\,\delta.
\end{align}
Furthermore, \eqref{gc} and Theorem \ref{thm.additiv_ergodic} give
\begin{align*}
	\limsup_{\e\to 0^+}F_{\e}(\w)(v_k,C^k_{r_h,s}(x_0))&\leq C (t_h^k k^{1-d}|\eta\otimes n|+1)|C^k_{r_h,s}(x_0)|
	\\
	&\leq Ct_h^k k^{1-d}s |C_{r_h}^k(x_0)|,
\end{align*}
where we also used that $t_h^k k^{1-d}\geq 1$ for $h=h(k)$ large enough. 

By the arbitrariness of $\delta>0$ in \eqref{eq:applying_fund_est}, appealing to the rank-one convexity of $f_{\rm hom}$ (see Step 2 in the proof of Proposition \ref{p.ub}) and to \eqref{eq:scale_t} we deduce
\begin{align*}
f_{\rm hom}^\infty(\eta\otimes n)=\lim_{h\to +\infty }\frac{f_{\rm hom}(t_{h}^k k^{1-d}\eta\otimes n)}{t_{h}^k k^{1-d}}&\leq\limsup_{h\to +\infty}\liminf_{\e\to 0^+}\frac{1}{t_h^k k^{1-d}|C^k_{r_h}(x_0)|}F_{\e}(\w)(u_{\e},C_{r_h}^k(x_0))+C s
\\
&\quad+\liminf_{h\to +\infty}\frac{C}{t_{r_h}^k k^{1-d}|C_{r_h}^k(x_0)|sr}\int_{C_{r_h,s}^k(x_0)}|u-v_k|\dx.
\end{align*}
By definition of $t_h^k$ there holds
\begin{equation*}
t_h^k k^{1-d}|C_{r_h}^k(x_0)|=|\D u|(C_{r_h}^k(x_0))k^{1-d}|C^k(x_0)|=|\D u|(C_{r_h}^k(x_0)),
\end{equation*}
therefore \eqref{eq:claim_singular} follows if we show that
\begin{equation}\label{eq:triplelimit}
\limsup_{k\to +\infty}\limsup_{s\to 0^+}\liminf_{h\to +\infty}\frac{1}{|\D u|(C_{r_h}^k(x_0))sr_h}\int_{C_{r_h,s}^k(x_0)}|u-v_k|\dx=0.
\end{equation}
To prove \eqref{eq:triplelimit} we start observing that $C_{r,s}^k=r C_{1,s}^k$, hence by a change of variables we get
\begin{equation*}
	\frac{1}{|\D u|(C_r^k(x_0))sr}\int_{C_{r,s}^k(x_0)}|u-v_k|\dx=\frac{r^d}{|\D u|(C_{r}^k(x_0))s}\int_{C^k_{1,s}}\left|\frac{u(x_0+ry)-v_k(x_0+ry)}{r}\right|\dy.
\end{equation*}
Moreover, in view of the definition of $v_k$ we have
\begin{align*}
\frac{r^d}{|\D u|(C_r^k(x_0))}\frac{u(x_0+ry)-v_k(x_0+ry)}{r}&=\frac{r^d}{|\D u|(C_r^k(x_0))}r^{-1}\Big(u(x_0+ry)-\dashint_{C_1^k}u(x_0+rx)\dx\Big)
\\
&\quad- \left(k^{1-d}\eta\langle n,y\rangle+\eta \frac{\psi^k(\tfrac{1}{2}^-)+\psi^k(-\tfrac{1}{2}^+)}{2}\right)
\\
&=w_r(y)- \left(k^{1-d}\eta\langle n,y\rangle+ \eta\frac{\psi^k(\tfrac{1}{2}^-)+\psi^k(-\tfrac{1}{2}^+)}{2}\right).
\end{align*}
Then, if we replace $r$ with the sequence $(r_h)$ chosen as above, we obtain that
\begin{align}
	&\liminf_{h\to +\infty}\frac{1}{t_h^k s}\int_{C^k_{1,s}}\left|\frac{u(x_0+r_h y)-v_k(x_0+r_h y)}{r_h}\right|\dy\nonumber
	\\
	=&\frac{1}{s}\int_{C^k_{1,s}}\left|\psi^k(\langle n,y\rangle)\eta-\left(k^{1-d}\eta\langle n,y\rangle+\eta \frac{\psi^k(\tfrac{1}{2}^-)+\psi^k(-\tfrac{1}{2}^+)}{2}\right)\right|\dy\nonumber
	\\
	=&\frac{1}{s}\int_{C^k_{1,s}}\left|\psi^k(\langle n,y\rangle)-\left(k^{1-d}\langle n,y\rangle+ \frac{\psi^k(\tfrac{1}{2}^-)+\psi^k(-\tfrac{1}{2}^+)}{2}\right)\right|\dy,\label{eq:should_vanish}
\end{align}
where we also used that $|\eta|=1$. 

We now estimate \eqref{eq:should_vanish}. To do so we split the domain of integration in two (non-disjoint) subsets by writing $C^k_{1,s}=A^k_{1,s}\cup B^k_{1,s}$ with
\begin{align*}
A^k_{1,s}&=\{y\in C^k:\,\max_{1\leq i\leq d-1}|\langle y, n_i\rangle|> (k-s)/2\},
\\
B^k_{1,s}&=\{y\in C^k:\,|\langle y,n\rangle|> (1-s)/2\}.
\end{align*}
Then, the idea to conclude is as follows: the measure of $A^k_{1,s}$ is of order $(d-1)sk^{d-2}$ while the integrand decays like $k^{1-d}$; on the other hand, in $B_{1,s}^k$ the quantity $\langle n,y\rangle$ is close to $\pm 1/2$ and, by construction, the integrand vanishes at these points. 

Rigorously, since 
\[
A_{1,s}^k=\bigcup_{i=1}^{d-1} C^k\cap \{y\colon |\langle y,n_i \rangle|\geq (k-s)/2\}
\]
we get that $|A^k_{1,s}|\leq (d-1)sk^{d-2}$, where we have used the fact that the vectors $(n_i)_{i=1}^{d-1}$ and $n$ form an orthonormal basis. Moreover, \eqref{eq:jump_amplitude}, the monotonicity of $\psi^k$, and \eqref{eq:meanzero} imply that $|\psi^k(t)|\leq k^{1-d}$ for all $t\in (-1/2,1/2)$. Hence 
\begin{equation}\label{eq:control_on_A}
	\frac{1}{s}\int_{A_{1,s}^k}\left|\psi^k(\langle n,y\rangle)-\left(k^{1-d}\langle n,y\rangle+ \frac{\psi^k(\tfrac{1}{2}^-)+\psi^k(-\tfrac{1}{2}^+)}{2}\right)\right|\dy\leq \frac{C}{s}|A^k_{1,s}|k^{1-d}\leq\frac{C}{k},
\end{equation}
uniformly for $s\in(0,1)$. We now estimate the contribution coming from the integration on $B_{1,s}^k$. By an orthogonal change of variables we get
\begin{align*}
&\frac{1}{s}\int_{B_{1,s}^k}\left|\psi^k(\langle n,y\rangle)-\left(k^{1-d}\langle n,y\rangle+ \frac{\psi^k(\tfrac{1}{2}^-)+\psi^k(-\tfrac{1}{2}^+)}{2}\right)\right|\dy
\\
=&\frac{k^{d-1}}{s}\int_{-\tfrac{1}{2}}^{-\tfrac{1-s}{2}}\left|\psi^k(t)-\left(k^{1-d}t+ \frac{\psi^k(\tfrac{1}{2}^-)+\psi^k(-\tfrac{1}{2}^+)}{2}\right)\right|\,\mathrm{d}t
\\
&+\frac{k^{d-1}}{s}\int_{\tfrac{1-s}{2}}^{\tfrac{1}{2}}\left|\psi^k(t)-\left(k^{1-d}t+ \frac{\psi^k(\tfrac{1}{2}^-)+\psi^k(-\tfrac{1}{2}^+)}{2}\right)\right|\,\mathrm{d}t=:I^-_{k,s}+I^+_{k,s}.
\end{align*}
Due the monotonicity of $\psi$ we find a sequence of positive numbers $(\gamma^k_s)_{s>0}$ with $\lim_{s\to 0^+}\gamma^k_s=0$ such that 
\begin{align*}
	&\left|\psi^k(t)-\psi^k(-\tfrac{1}{2}^+)\right|\leq\gamma_s^k\quad\text{ for all }-\tfrac{1}{2}<t<-\tfrac{1-s}{2},
	\\
	&\left|\psi^k(t)-\psi^k(\tfrac{1}{2}^-)\right|\leq\gamma_s^k\qquad\text{ for all }\tfrac{1-s}{2}<t<\tfrac{1}{2}.
\end{align*}
Hence we can estimate the last two integrals as follows:
\begin{align*}
I^-_{k,s}+I^+_{k,s}&\leq k^{d-1}\gamma^k_s+\frac{s}{2}+\frac{k^{d-1}}{2}\left|\psi^k(-\tfrac{1}{2}^+)-\left(-\frac{k^{1-d}}{2}+ \frac{\psi^k(\tfrac{1}{2}^-)+\psi^k(-\tfrac{1}{2}^+)}{2}\right)\right|
\\
&\quad +\frac{k^{d-1}}{2}\left|\psi^k(\tfrac{1}{2}^-)-\left(\frac{k^{1-d}}{2}+ \frac{\psi^k(\tfrac{1}{2}^-)+\psi^k(-\tfrac{1}{2}^+)}{2}\right)\right|.
\end{align*}
By \eqref{eq:jump_amplitude} the last two terms equal zero and we obtain
\begin{equation*}
\limsup_{s\to 0^+}\left(I^-_{k,s}+I^+_{k,s}\right)\leq \limsup_{s\to 0^+}\left(k^{d-1}\gamma^k_s+\frac{s}{2}\right)=0.
\end{equation*}
By combining the latter with \eqref{eq:control_on_A} and \eqref{eq:should_vanish} we get  \eqref{eq:triplelimit} and hence the claim.

\medskip

{\bf Step 2:} \emph{Proof of \eqref{c:li} for general sequences $(u_\e)$.} 

\smallskip

Let $(u_{\e})\subset W^{1,1}(A,\R^m)$ be as in \eqref{c:li-bd}. For $\e>0$ and $\eta>0$ fixed let $u_{\e,\eta}\in W^{1,1}(A,\R^m)$ be the function given by Lemma \ref{l.truncation}; therefore
\begin{equation}\label{c:uffa}
	\liminf_{\e\to 0^+}F_{\e}(\w)(u_{\e},A)\geq \frac{1}{1+\eta}\liminf_{\e\to 0^+}F_{\e}(\w)(u_{\e,\eta},A)-\eta.
\end{equation}
Since $|u_{\e,\eta}|<C_\eta$ a.e. in $A$ , thanks to \eqref{c:uffa} we can invoke Lemma \ref{l.compactness} to deduce the existence of a subsequence (not relabelled) such that $u_{\e,\eta}\to u_{\eta}$ in $L^1(A,\R^m)$, as $\e \to 0^+$, for some $u_{\eta}\in BV(A,\R^m)$. Moreover, Step 1 implies that
\begin{equation}\label{eq:eta_inequality}
	\liminf_{\e\to 0^+}F_{\e}(\w)(u_{\e},A)\geq \frac{1}{1+\eta}F_{\rm hom}(u_{\eta},A)-\eta.
\end{equation} 
Thanks to \cite[Theorem 4.1]{AmDM92} the functional $F_{\rm hom}(\cdot,A)$
is $L^1(A,\R^m)$-lower semicontinuous on $BV(A,\R^m)$. Therefore, in view of \eqref{eq:eta_inequality}, to conclude it suffices to show that $u_{\eta}\to u$ in $L^1(A,\R^m)$. 

By Lemma \ref{l.truncation} and the $L^1$-convergence of $u_{\e,\eta}$ to $u_{\eta}$ it follows that $|u_{\eta}|\leq |u|$ a.e. on $A$; moreover by construction $u_{\eta}=u$ a.e. on $\{|u|\leq\eta^{-1}\}$. Hence the Dominated Convergence Theorem yields $u_{\eta}\to u$ as $\eta\to 0^+$ and thus the claim. 
\end{proof}

By combining the results proven in this section together with those in Section \ref{sec:existence} we are now able to prove Theorem \ref{thm.Gamma_pure}.

\begin{proof}[Proof of Theorem \ref{thm.Gamma_pure}]
Lemma \ref{l.compactness} shows that the domain of the $\Gamma$-limit of $F_\e(\omega)(\cdot, A)$ is $BV(A,\R^m)$. Then, statements in $i.$ and $ii.$ are proven in Lemma~\ref{l.existence_f_hom} and in Proposition \ref{p.ub}, Step 2. Eventually, the almost sure $\Gamma$-convergence of the functionals $F_\e(\omega)(\cdot, A)$, statement $iii.$, follows by Propositions \ref{p.ub} and \ref{p.lb}. 	
\end{proof}

\section{$\Gamma$-convergence with Dirichlet boundary conditions}\label{s.bc}

This short section is devoted to the proof of Theorem \ref{thm:bc}.

\begin{proof}[Proof of Theorem \ref{thm:bc}]
Let $\widetilde \Omega \in \F$ be as in Theorem \ref{thm.Gamma_pure} and $\widehat \Omega$ be such that the sequence of functions $(M_{\e}(\w))_\e$ defined in \eqref{eq:boundarymeasure} is locally equi-integrable in $\R^d$, for every $\omega\in \widehat \Omega$.
Throughout the proof $\omega$ is arbitrarily fixed in $\widetilde \Omega \cap \widehat \Omega$. 

We start by proving the liminf-inequality. To this end, fix $A\in \A$ and let $(u_{\e}) \subset L^1(A,\R^m)$ and $u\in L^1(A,\R^m)$
be such that $u_\e \to u$ in $L^1(A,\R^m)$.
Without loss of generality, we can assume that 
\begin{equation*}
\sup_{\e>0}F_{\e}^{u_0}(\w)(u_{\e},A)<+\infty,
\end{equation*}
therefore $(u_\e) \subset u_0 + W^{1,1}_0(A,\R^m)$. We then extend $u_{\e}$ to the whole $\R^d$ by setting $u_{\e}:=u_0$ on $\R^d\setminus A$. 

For $r\in (0,1)$ given, consider the sets $A_r:=\{x\in\R^d\setminus A:\,\dist(x,A)< r\}$. By \eqref{gc} we get
\begin{equation*}
F_{\e}(\w)(u_{\e},A\cup A_r)\leq F^{u_0}_{\e}(\w)(u_{\e},A)+F_{\e}^{u_0}(\w)(u_0,A_r)\leq F_{\e}^{u_0}(\w)(u_{\e},A)+\int_{A_r}\big(M_{\e}(\w)(x)+\lambda(\w,\tfrac{x}{\e})\big)\dx.
\end{equation*}
By virtue of the equi-integrability of $M_{\e}(\w)$ and $\lambda(\w,\cdot/\e)$ (cf. Theorem \ref{thm.additiv_ergodic}),
given $\delta>0$,  there exists $r_\delta>0$ such that for every $r\in (0,r_\delta)$
\begin{equation*}
\int_{A_r}\big(M_{\e}(\w)(x)+\lambda(\w,\tfrac{x}{\e})\big)\dx\leq\delta,
\end{equation*}
for every $\e>0$. 
Then, setting $\tilde{u}:=\chi_A u+(1-\chi_A)u_0$, by applying Theorem \ref{thm.Gamma_pure} in the open set $A\cup A_r$ (notice that $A\cup A_r\in \A$, for $r$ small) we obtain 
\begin{equation*}
\int_{A\cup A_r} f_{\rm hom}(\nabla \tilde{u})\dx+\int_{A\cup A_r}f_{\rm hom}^{\infty}\left(\frac{\mathrm{d}D^s\tilde{u}}{\mathrm{d}|D^s\tilde{u}|}\right)\,\mathrm{d}|D^s\tilde{u}|\leq\liminf_{\e\to 0^+}F_{\e}^{u_0}(\w)(u_{\e},A)+\delta,
\end{equation*}
for every $r\in (0,r_\delta)$. Therefore letting first $r\to 0^+$ and then $\delta\to 0^+$, we infer that
\begin{equation*}
\int_{A} f_{\rm hom}(\nabla u)\dx+\int_{\overline{A}}f_{\rm hom}^{\infty}\left(\frac{\mathrm{d}D^s\widetilde{u}}{\mathrm{d}|D^s\widetilde{u}|}\right)\,\mathrm{d}|D^s\widetilde{u}|\leq\liminf_{\e\to 0^+}F^{u_0}_{\e}(\w)(u_{\e},A),
\end{equation*}
hence the claim follows by the one-homogeneity of the recessions function $f_{\rm hom}^{\infty}$ and by \cite[Corollary 3.89]{AmbFusPal}.

\medskip

We now turn to the proof of the limsup-inequality. 

We start observing that it is not restrictive to assume that target function $u\in BV(A,\R^m)$ satisfies $u=u_0$ in a neighbourhood of $\partial A$. Indeed, this can be achieved by the following approximation argument. For $u\in BV(A,\R^m)$ fixed, define $\tilde{u}:=\chi_A u+(1-\chi_A)u_0$; then by combining \cite[Theorem 1.2]{Schm} and \cite[Theorem 3]{KrRi} we can deduce the existence of a sequence $(u_n)\subset u_0+C_c^{\infty}(A,\R^m)$ such that $u_n\to \tilde{u}$ in $L^1(\R^d,\R^m)$ and
\begin{equation*}
\lim_{n\to +\infty}F_{\rm hom}(u_n,A')=F_{\rm hom}(\tilde{u},A'),
\end{equation*}
for every $A'\in \A$ with $\overline{A} \subset A'$.
Now assume that the $\Gamma$-limsup inequality holds true for $u_n$, that is 
\[
(F^{u_0})''(\omega)(u_n,A) \leq F^{u_0}_{\rm hom}(u_n,A), 
\]
for every $n\in \N$, where $(F^{u_0})''$ denotes the $\Gamma$-limsup of $F_\e^{u_0}$ (cf. \eqref{F''}). 
Then by the lower semicontinuity of $(F^{u_0})''(\omega)$ we get
\begin{align*}
(F^{u_0})''(\omega)(u,A)&\leq \liminf_{n\to+\infty}(F^{u_0})''(\omega)(u_n,A)\leq \liminf_{n\to +\infty}F^{u_0}_{\rm hom}(u_n,A)
\\
&\leq \lim_{n\to +\infty}F^{u_0}_{\rm hom}(u_n,A')=F^{u_0}_{\rm hom}(\tilde{u},A').
\end{align*}
Then, the claim follows by letting $A'\searrow \overline{A}$, again using the one-homogeneity of $f_{\rm hom}^{\infty}$ and \cite[Corollary 3.89]{AmbFusPal}.

Hence we only need to prove the upper bound inequality for those target functions $u$ belonging to $u_0+C_c^{\infty}(A,\R^m)$. 
To do so, we need to modify a recovery sequence given by Theorem \ref{thm.Gamma_pure} close to $\partial A$, in order to satisfy the correct trace-constraint. 

Let $\eta>0$, Lemma \ref{l.truncation} applied to a recovery sequence $(u_{\e})$ for $u$ yields a sequence $(u_{\e,\eta})$ which is bounded in $L^{\infty}(\R^d,\R^m)$ uniformly in $\e$ and such that $u_{\e,\eta}=u_{\e}$ on $\{|u_{\e}|\leq\eta^{-1}\}$. Moreover, it satisfies 
\begin{equation}\label{eq:trunc_finite_energy}
\limsup_{\e\to 0^+}F_{\e}(\w)(u_{\e,\eta},A)\leq (1+\eta)F_{\rm hom}(u,A)+\eta,
\end{equation}  
where we have also used the limsup-inequality in Theorem \ref{thm.Gamma_pure}.
Next we apply Lemma \ref{l.fundamental_estimate} with $u=u_{\e,\eta}$, $v=u_0$, $A''=A$ and $B=A\setminus \overline{A'}$, where $A'\in \A$ is such that $u=u_0$ in $A\setminus \overline{A'}$. Hence, for any $\delta>0$ we obtain a sequence $(w_{\e,\eta,\delta})$ such that $w_{\e,\eta,\delta}=u_0$ in a neighbourhood of $\partial A$, $w_{\e,\eta,\delta}= u_{\e,\eta}$ in $\overline{A'}$, and
\begin{align}\label{eq:big_est}
& \; \limsup_{\e\to 0^+}F_{\e}^{u_0}(\w)(w_{\e,\eta,\delta},A)=\limsup_{\e\to 0^+}F_{\e}(\w)(w_{\e,\eta,\delta},A)\nonumber
\\
&\leq (1+\delta)\Big((1+\eta)F_{\rm hom}(u,A)+\eta+\limsup_{\e\to 0^+}F_{\e}(\w)(u_0,A\setminus\overline{A'})\Big)\nonumber
\\
&\quad +\limsup_{\e\to 0}\frac{C}{\dist(A',\partial A)}\int_{A\setminus \overline{A'}}|u_{\e,\eta}-u_0||\Am(\w,\tfrac{x}{\e})|\dx+C\delta|A\setminus\overline{A'}|.
\end{align}
where we tacitly used that $|\partial A'|=0$. Since
\begin{equation*}
\limsup_{\e\to 0^+}F_{\e}(\w)(u_0,A\setminus\overline{A'})\leq \limsup_{\e\to 0^+}\int_{A\setminus \overline{A'}}\big(M_{\e}(\w)(x)+\lambda(\w,\tfrac{x}{\e})\big)\dx,
\end{equation*}
the equi-integrability of $M_\e(\omega)+\lambda(\w,\tfrac{\cdot}{\e})$ again implies that the term above becomes arbitrarily small when $A' \nearrow A$, independently of $\eta$ and $\delta$. 

We now estimate the integral in \eqref{eq:big_est}. Choosing a subsequence which realises the $\limsup$, Lemma \ref{l.compactness} together with \eqref{eq:trunc_finite_energy} implies that (up to a further subsequence) there holds 
$u_{\e,\eta}\to u_{\eta}$ in $L^1(A,\R^m)$ and a.e. in $A$, for some $u_{\eta}\in L^{\infty}(A,\R^m)$ satisfying $u_{\eta}=u=u_0$ a.e. on $\{|u_0|<\eta^{-1}\}\cap A\setminus\overline{A'}$. 
We claim that Assumption \ref{a.2} implies that (along that subsequence) 
\begin{equation}\label{c:final}
|u_{\e,\eta}-u_0||\Am(\w,\tfrac{x}{\e})|\rightharpoonup \mathbb{E}[|\Am(\cdot,0)]|u_{\eta}-u_0|,
\end{equation}
in $L^1(A)$. Indeed, for $k\in \N$ set $L_k:=\{|u_0|\leq k\}$; then the sequence $|u_{\e,\eta}-u_0|$ is bounded in $L^{\infty}(A\cap L_k)$ and converges a.e. to $|u_{\eta}-u_0|$, while by Theorem \ref{thm.additiv_ergodic} the sequence $|\Am(\w,\tfrac{\cdot}{\e})|$ converges weakly in $L^1(A\cap L_k)$ to $\mathbb{E}[|\Am(\cdot,0)|]$. Hence from \cite[Proposition 2.61]{FoLe} we infer that
\begin{equation*}
	|u_{\e,\eta}-u_0||\Am(\w,\tfrac{\cdot}{\e})|\rightharpoonup \mathbb{E}[|\Am(\cdot,0)]|u_{\eta}-u_0|\text{ in }L^1(A\cap L_k).
\end{equation*}
Moreover, since 
\begin{equation*}
	|u_{\e,\eta}-u_0||\Am(\w,\tfrac{x}{\e})|\leq \sup_{\e}\|u_{\e,\eta}\|_{L^{\infty}(A)}|\Am(\w,\tfrac{x}{\e})|+M_{\e}(\w)(x),
\end{equation*}
Assumption \ref{a.2} and Theorem \ref{thm.additiv_ergodic} imply that $|u_{\e,\eta}-u_0||\Am(\w,\tfrac{x}{\e})|$ is bounded in $L^1(A)$ and equi-integrable. 

Therefore, since by definition of $L_k$ we have that $|A\setminus L_k|\to 0$ when $k\to +\infty$, we can easily deduce \eqref{c:final}.  
Hence, we obtain that 
\begin{equation*}
	\limsup_{\e\to 0^+}\frac{1}{\dist(A',\partial A)}\int_{A\setminus \overline{A'}}|u_{\e,\eta}-u_0||\Am(\w,\tfrac{x}{\e})|\dx\leq \frac{C}{\dist(A',\partial A)}\int_{A\setminus \overline{A'}}|u_{\eta}-u_0|\dx.
\end{equation*}
Recalling that on $A\setminus\overline{A'}$ we have $|u_{\eta}|\leq |u|=|u_0|$, the Dominated Convergence Theorem ensures that the right-hand side in the expression above vanishes as $\eta\to 0^+$.

Eventually, letting first $\delta\to 0^+$, then $\eta\to 0^+$, and finally $A' \nearrow A$ in \eqref{eq:big_est}, we obtain
\begin{equation*}
	\limsup_{\eta\to 0}\limsup_{\delta\to 0}\limsup_{\e\to 0}F_{\e}^{u_0}(\w)(w_{\e,\eta,\delta},A)\leq F_{\rm hom}(u,A)=F_{\rm hom}^{u_0}(u,A).
\end{equation*}
In order to conclude we need to estimate the difference between $u$ and $w_{\e,\eta,\delta}$. To this end, we recall that $w_{\e,\eta,\delta}$ is given by a convex combination of $u_{\e,\eta}$ and $u_0$ with a cut-off function and that $w_{\e,\eta,\delta}=u_{\e,\eta}$ in $\overline{A'}$. Then, arguing as above it can be easily shown that $\lim_{\eta\to 0^+}\limsup_{\e\to 0^+}\|u_{\e,\eta}-u\|_{L^1(A)}=0$. Hence we have
\begin{align*}
	\limsup_{\eta\to 0^+}\limsup_{\e\to 0^+}\|u-w_{\e,\eta,\delta}\|_{L^1(A)}&\leq \limsup_{\eta\to 0^+}\limsup_{\e\to 0^+}\|u_{\e,\eta}-w_{\e,\eta,\delta}\|_{L^1(A)}
	\\
	&\leq \limsup_{\eta\to 0^+}\limsup_{\e\to 0^+}\|u_{\e,\eta}-u_0\|_{L^1(A\setminus\overline{A'})}=\|u-u_0\|_{L^1(A\setminus\overline{A'})}=0.
\end{align*}
Hence by a standard diagonal argument we get
\begin{equation*}
(F^{u_0})''(\omega)(u,A)\leq F^{u_0}_{\rm hom}(u,A),
\end{equation*}
and hence the limsup-inequality.
\end{proof}

\appendix

\section{Measurability}\label{app:2}
Let $\xi\in \R^{m\times d}$ and $A\in \A$ be fixed; this last section is devoted to the proof of the $\F$-measurability of $\omega \mapsto \mu_{\xi}(\w,A)$, defined as in \eqref{c:mu-xi}.

\begin{lemma}\label{l.measurable}
Let $f$ satisfy Assumption \ref{a.1} and for $\xi\in \R^{m\times d}$ let $\mu_{\xi}$ be defined as in \eqref{c:mu-xi}. Then $\omega \mapsto \mu_{\xi}(\w,A)$ is $\F$-measurable for every $\xi \in \R^{m\times d}$ and for every $A\in \A$.
\end{lemma}
The proof of Lemma \ref{l.measurable} relies on the following abstract measurability result proven in \cite[Lemma C.2]{RR21}.
\begin{lemma}\label{l.oninf}
Let $Y$ be a complete, separable metric space and let $\mathcal{B}(Y)$ denote the Borel $\sigma$-algebra on $Y$. Let $(X,\mathcal{T},m)$ be a complete measure space and assume that $F:X\times Y\to \R\cup\{+\infty\}$ is $\mathcal{T}\otimes \mathcal{B}(Y)$-measurable and that $y\mapsto F(x,y)$ is lower semicontinuous and not constantly equal to $+\infty$ for every $x\in X$. Then the function $x\mapsto\inf_{y\in Y} F(x,y)$ is $\mathcal{T}$-measurable.
\end{lemma}

\begin{proof}[Proof of Lemma \ref{l.measurable}]
We recall that by assumption $(\Omega,\F,\mathbb P)$ is a complete probability space, thus we can set $f(\w,x,\xi)=|\xi|$ on the set with zero probability where $|\Lambda(\w,\cdot)|+\lambda(\w,\cdot)$ is not locally integrable and this modification does not affect the measurability of $f$. We now ``regularise'' the integrand $f$ in the variable $\xi$ by considering its Moreau-Yosida transform. That is, for every $k\in \N$ we consider the function defined as
\begin{equation*}
f_k(\w,x,\xi):=\inf_{\zeta\in\R^{m\times d}}\{f(\w,x,\zeta)+k|\zeta-\xi|\};
\end{equation*}
it is well-known that $f_k$ is $k$-Lipschitz in $\xi$. 

We now want to apply Lemma \ref{l.oninf} with $X=\Omega \times \R^d$ and $Y=\R^{m\times d}$. To do so, 
we need to complete the product $\sigma$-algebra $\mathcal{F}\otimes\mathcal{L}^d$ with respect to the product measure $\mathbb{P}\times d\mathcal L^d$; we denote this completion by $\overline{\mathcal{F}\otimes\mathcal{L}^d}$. Then,
considering $\xi$ as a parameter, by Assumption \ref{a.1} and Lemma \ref{l.oninf} we deduce that the function $(\w,x)\to f_k(\w,x,\xi)$ is $\overline{\mathcal{F}\otimes\mathcal{L}^d}$-measurable. Hence, by a well-known property of Carath\'eodory functions it follows that $f_k$ is $\overline{\mathcal{F}\otimes\mathcal{L}^d}\otimes\mathcal{B}(\R^{m\times d})$-measurable. Therefore, for every fixed $u\in \ell_{\xi}+W_0^{1,1}(A,\R^m)$ the function $(\w,x)\mapsto f_k(\w,x,\nabla u(x))$ is $\overline{\mathcal{F}\otimes\mathcal{L}^d}$-measurable and, by \cite[Theorem 1.121]{FoLe} we can define $F^k:\Omega\times \big(\ell_\xi+W_0^{1,1}(A,\R^m)\big)\longrightarrow [0,+\infty)$ as
\begin{equation*}
F^k(\w)(u)=\int_A f_k(\w,x,\nabla u)\,\dx.
\end{equation*}
We observe that the integral above is finite; indeed the non-negativity of $f$ and the Lipschitz-continuity of $f_k$ imply that
\begin{equation}\label{eq:lineargrowth}
0\leq f_k(\w,x,\xi)\leq  f_k(\w,x,0)+k|\xi|\leq f(\w,x,0)+k|\xi|\leq \lambda(\w,x)+k|\xi|,	
\end{equation}
where $\lambda(\w,\cdot)\in L^1_{\rm loc}(\R^d)$.
Moreover, thanks to \eqref{eq:lineargrowth} and to the Lipschitz-continuity of $f_k$ in $\xi$, the functional $F^k(\w)$ is continuous on $\ell_\xi+W_0^{1,1}(A,\R^m)$. Then, for fixed $u\in\ell_\xi+W^{1,1}_0(A,\R^m)$, by Tonelli's Theorem the function $\w\mapsto F^k(\w)(u)$ is $\F$-measurable; thus, in particular, $F^k$ is $\mathcal{F}\otimes\mathcal{B}(\ell_\xi+W^{1,1}_0(A,\R^m))$-measurable. 

Furthermore, thanks to the lower semicontinuity of $f$ we have that $f_k\nearrow f$ pointwise. Therefore we get that for every $u\in \ell_\xi+W^{1,1}_0(A,\R^m)$
\[
F^k(\omega)(u)\to \int_A f(\w,x,\nabla u)\,\dx,
\] 
as $k\to +\infty$; the latter together with the measurability of $(\omega,u) \mapsto F_k(\omega)(u)$ ensure that 
\[
(\omega,u) \mapsto \int_A f(\w,x,\nabla u)\,\dx
\]
is $\mathcal{F}\otimes\mathcal{B}(\ell_\xi+W^{1,1}_0(A,\R^m))$-measurable (and not constantly equal to $+\infty$ for fixed $\w\in\Omega$) 

Eventually, since $W_0^{1,1}(A,\R^m)$ is a separable, complete metric space we can again apply Lemma \ref{l.oninf} now choosing $Y=\ell_\xi+W_0^{1,1}(A,\R^m)$ and $X=\Omega$ to deduce the $\F$-measurability of
\[
\omega \; \mapsto \hspace{-0.2cm}\inf_{\ell_\xi+W^{1,1}_0(A,\R^m)} \int_A f(\w,x,\nabla u)\,\dx
\]
and hence the desired result. 
\end{proof}

\section*{Acknowledgments}
The work of C.I.Z. was supported by the Deutsche Forschungsgemeinschaft (DFG, German Research Foundation) under the Germany Excellence Strategy EXC 2044-390685587 Mathematics M\"unster: Dynamics--Geometry--Structure.


\begin{thebibliography}{99}

\bibitem{AML} Y.~Abddaimi, G.~Michaille and C.~Licht, 
\newblock Stochastic Homogenization for an integral functional of a quasiconvex function with linear growth.
\newblock {\em Asymptot. Anal.}, {\bf 15} (1997), 183--202.

\bibitem{AkKr} 
\newblock M.A. Akcoglu and U. Krengel,
\newblock Ergodic theorems for superadditive processes.
\newblock \emph{J. Reine Ang. Math.}, \textbf{323} (1981) , 53--67.
%
%
%

\bibitem{AmDM92}
L.~Ambrosio and G.~Dal~Maso,
\newblock On the relaxation in $BV(\Omega;\R^m)$ of quasi-convex integrals,
\newblock {\em J. Funct. Anal.}, {\bf 109} (1992), 76--97.
%
\bibitem{AmbFusPal} {L. Ambrosio, N. Fusco and D. Pallara}. {\em Functions of Bounded Variation and Free Discontinuity Problems}, Clarendon Press Oxford, 2000.
%
%
%











\bibitem{BFM}
G.~Bouchitt\`e, I.~Fonseca and L.~Mascarenhas, 
	\newblock{A global method for relaxation}, 
	\newblock \emph{Arch. Ration. Mech. Anal.}, {\bf 145} (1998), 51--98.
	
\bibitem{BrDf}{\sc A.~Braides, A.~Defranceschi},
\newblock \emph{Homogenization of Multiple Integrals},
\newblock{Oxford University Press, New York, 1998}.

\bibitem{Br}
A.~Braides,
\newblock {\em {$\Gamma$}-convergence for beginners}, Oxford
Lecture Series in Mathematics and its Applications, vol.22,
\newblock Oxford University Press, Oxford, 2002.

\bibitem{CDMSZ21} 
	F.~Cagnetti, G.~Dal~Maso, L.~Scardia and C.~I.~Zeppieri,
	\newblock {A global method for deterministic and stochastic homogenization in BV}
	\newblock {\em Ann. PDE},  {\bf 8} no. 1 (2022), Paper No. 8, 89 pp.


\bibitem{CDMSZ18}
F.~Cagnetti, G.~Dal~Maso, L.~Scardia and C.~I.~Zeppieri,
\newblock {$\Gamma$-convergence of free discontinuity problems}.
\newblock {\em Ann. Inst. H. Poincar\'e Anal. Non Lin\'eaire} {\bf 36} (2019), 1035--1079.


\bibitem{CDMSZ19}
F.~Cagnetti, G.~Dal~Maso, L.~Scardia and C.~I.~Zeppieri,
\newblock {Stochastic homogenization of free discontinuity problems}.
\newblock {\em Arch. Ration. Mech. Anal.} {\bf 233} (2019), 935--974.
	


\bibitem{DM}
G.~Dal~Maso,
\newblock {\em An introduction to {$\Gamma$}-convergence},
\newblock Progress in Nonlinear Differential Equations and their Applications,
vol. 8, Birkh{\"a}user Boston Inc., Boston, MA, 1993.

\bibitem{DMMoIII}
G.~Dal~Maso and L.~Modica, Integral functionals determined by their minima,
\newblock {\em Rend. Sem. Mat. Univ. Padova} {\bf 76} (1986), 255--267.	
	  
\bibitem{DMMoI} 
G.~Dal~Maso and L.~Modica, Nonlinear stochastic homogenization,
\newblock {\em Ann. Mat. Pura Appl.} {\bf 144} (1986), 347-–389,.
%
\bibitem{DMMoII} 
G.~Dal~Maso and L.~Modica, Nonlinear stochastic homogenization and ergodic theory,
\newblock {\em J.Reine Angew. Math.}, {\bf 368} (1986), 28–42.	
	
	





\bibitem{D'OnZe}
C. D'Onofrio and C.~I.~Zeppieri,
\newblock $\Gamma$-convergence and stochastic homogenization of degenerate integral functionals in weighted Sobolev spaces,
\newblock {\em Proc. Edinb. Math. Soc. (2)}, online first (2022), doi:10.1017/prm.2022.3

%


\bibitem{EvGa}
L.~C.~Evans, R.~F.~Gariepy,
\newblock {\em Measure theory and fine properties of functions}, Studies in advanced mathematics,
\newblock CRC Press Inc., Boca Raton, Florida, 1992.
%
\bibitem{FoLe} 
\newblock I. Fonseca and G. Leoni,
\newblock \emph{Modern Methods in the Calculus of Variations: $L^p$ spaces},
\newblock Springer, New York, 2007.
%
\bibitem{FoMue} I. Fonseca, S. M\"uller, 
\newblock Quasiconvex integrands and lower semicontinuity in $L^1$, 
\newblock{\it SIAM J. Math. Anal.} {\bf 23} (1992), 1081--1098. 
%
\bibitem{JKO}
V.~V.~Jikov, S.~M.~Kozlov and O.~A.~Oleinik,
\newblock {\em Homogenization of differential operators and integral functionals}, Springer-Verlag, Berlin, 1994.
%
\bibitem{Kr}
\newblock U.~Krengel, 
\newblock {\em Ergodic theorems}, De Gruyter studies in mathematics 6, 
\newblock De Gruyter, Berlin-New York, 1985.
%

\bibitem{KrRi}
J.~Kristensen and F.~Rindler,
\newblock Relaxation of signed integral functionals in $BV$,
\newblock {\em Calc. Var. Partial Differential Equations} {\bf 37} (2010), 29--62.


\bibitem{L}
C.~J.~Larsen,
\newblock Quasiconvexification in $W^{1,1}$ and optimal jump microstructure in $BV$ relaxation,
\newblock {\em SIAM J. Math. Anal.}, {\bf 29} (1998), 823--848.

\bibitem{LICHT-MICHAILLE} 
C.~Licht, G.~Michaille,
	\newblock Global-Local subadditive ergodic theorems and application to homogenization in elasticity,
	\newblock \emph{Ann. Math. Blaise Pascal} {\bf 9} (2002), 21--62. 

 
\bibitem{MeMi}
K.~Messaoudi and G.~Michaille,
\newblock Stochastic homogenization of nonconvex integral functionals,
\newblock {\em ESAIM Math. Model. Numer. Anal.}, {\bf 28} (1994), 329--356.

 

\bibitem{NSS}
S.~Neukamm, M.~Sch\"affner and A.~Schl\"omerkemper,
\newblock Stochastic homogenization of nonconvex discrete energies with degenerate growth,
\newblock {\em SIAM J. Math. Anal.}, {\bf 49} (2017), 1761--1809.

\bibitem{Schm} T.~Schmidt,
\newblock Strict interior approximation of sets of finite perimeter and functions of bounded variation,
\newblock {\em Proc. Amer. Math. Soc.} {\bf 143} (2015), 2069--2084.

%




\bibitem{RR21} M.~Ruf and T.~Ruf,
\newblock Stochastic homogenization of degenerate integral functionals and their Euler-Lagrange equations,
\newblock Preprint (2021), arXiv 2109.13013.
 
 
\end{thebibliography}
\end{document}